\documentclass[11pt,times,letter]{article}
\usepackage[utf8]{inputenc}
\usepackage{newunicodechar}
\newunicodechar{，}{,}
\usepackage{amsmath, amsfonts,amsthm}
\usepackage{dsfont}
\usepackage{bbm}

\usepackage[english]{babel}
\usepackage[normalem]{ulem}

\usepackage{latexsym,amssymb,graphicx}
\usepackage{verbatim}
\usepackage{mathrsfs}
\usepackage{epsfig}

\usepackage{xcolor}
\usepackage{tikz}
\usepackage{pgfplots}
\usepgfplotslibrary{fillbetween}

\usepackage{url}
\usepackage{bm}
\usepackage[numbers]{natbib}

\usepackage[colorlinks,linkcolor=blue,citecolor=blue,anchorcolor=blue]{hyperref}

\newtheorem{theorem}{Theorem}[section]
\newtheorem{corollary}[theorem]{Corollary}
\newtheorem{lemma}[theorem]{Lemma}
\newtheorem{proposition}[theorem]{Proposition}

\newtheorem{assumption}[theorem]{Assumption}
\newtheorem{definition}[theorem]{Definition}
\newtheorem{remark}[theorem]{Remark}
\newtheorem{example}[theorem]{Example}

\numberwithin{equation}{section}

\definecolor{Green}{rgb}{0.0, 0.4, 0.0}

\definecolor{Gray}{rgb}{0.7421875,0.7421875,0.7421875}

\definecolor{mycolor1}{rgb}{0.00000,0.44700,0.74100}%
\definecolor{mycolor2}{rgb}{0.85000,0.32500,0.09800}%
\definecolor{mycolor3}{rgb}{0.92900,0.69400,0.12500}%
\definecolor{mycolor4}{rgb}{0.49400,0.18400,0.55600}%
\definecolor{mycolor5}{rgb}{0.46600,0.67400,0.18800}%
\definecolor{mycolor6}{rgb}{0.00000,1.00000,1.00000}%

\allowdisplaybreaks[4]
\pgfplotsset{compat=newest}

\def\beq{\begin{eqnarray}}
\def\eeq{\end{eqnarray}}
\def\be*{\begin{eqnarray*}}
\def\ee*{\end{eqnarray*}}

\def \E{\mathbb{E}}
\def \F{\mathbb{F}}

\def \N{\mathbb{N}}

\def \P{\mathbb{P}}
\def \Q{\mathbb{Q}}
\def \R{\mathbb{R}}

\def \T{\mathbb{T}}

\def\Fc{{\cal F}}

\def\Jc{{\cal J}}

\def\Nc{{\cal N}}

\def\Rc{{\cal R}}

\def \d {\mbox{d}}
\def \eps {\varepsilon}

\def \r {\rho}

\def\E{\mathbb{E}}

\def\P{\mathbb{P}}

\def\TT{\mathcal{T}}
\def\R{\mathbb{R}}

\def\Fc{\mathcal{F}}

\def\N{\mathbb{N}}
\def\d{\mathrm{d}}

\def\namedlabel#1#2{\begingroup
	#2%
	\def\@currentlabel{#2}%
	\phantomsection\label{#1}\endgroup
}

\def \1 {{\mathbf 1}}

\title{\Large \bf
Distribution-constrained optimal multiple stopping: the Root-type solution
}
\date{\vspace{-5ex}}

\author{Shuoqing Deng
\thanks{
{\small The Hong Kong University of Science and Technology, Department of Mathematics (masdeng@ust.hk). S. Deng is supported by the Hong Kong University of Science and Technology Start-up Grant No. R9826 and Hong Kong RGC Early Career Scheme (ECS) Grant No. 26307125. }}%
\and
Daxin Huang
\thanks{
{\small
The Hong Kong University of Science and Technology, Department of Mathematics (dhuangax@connect.ust.hk).
}}
}

\begin{document}


\maketitle

\begin{abstract}

We consider the distribution-constrained optimal stopping problem introduced by Bayraktar and Miller \cite{BM2019} (Mathematical Finance, 2019) and Beiglb\"ock et al. \cite{BEES2018} (PTRF, 2018).
Motivated by the multi-marginal Skorokhod embedding problem (SEP) and applications in financial mathematics, we generalize the \textit{Root-type} solution to the multi-marginal case. The key difficulty is that the associated stopping barriers need not be ordered, so the problem in general cannot be reduced to a sequence of independent one-marginal problems. First, we give a probabilistic characterization in terms of sequential optimal stopping for an auxiliary backward process, in the same spirit as \cite{COT}.  Then, we prove the optimality of this construction by martingale methods for a general class of reward functions, including multi-marginal versions of all examples in \cite{BEES2018} as the special cases.

\vspace{0.6 cm}


\vspace{0.2 cm}


\end{abstract}

\section{Introduction}

Given a probability measure $\mu$ on $\R_+$ with finite first moment, the \textit{distribution-constrained optimal stopping problem} consists of finding a stopping time $\tau$ for Brownian motion $B$ such that $\tau \sim \mu$ and $\tau$ maximizes or minimizes the objective
$$
\E[c(B_{\cdot \wedge \tau}, \tau)],
$$
for some measurable reward function $c$.

This problem was first proposed by Bayraktar and Miller \cite{BM2019} in the atomic case. They transformed the problem into a state-constrained stochastic control problem with respect to an auxiliary process, and used the dynamic programming principle (DPP) to solve it. 
K\"allblad \cite{K2022} then extended it to a more general setting with a nonatomic distribution $\mu$. The key idea is to reformulate the problem using a measure-valued martingale (MVM), so that the original terminal condition $\tau \sim \mu$ becomes an initial condition in the lifted formulation and the DPP is restored.

Our paper follows the approach to distribution-constrained optimal stopping developed by Beiglb\"ock et al. \cite{BEES2018}, which was established around the same time as \cite{BM2019}. More precisely, they borrowed ideas from the optimal Skorokhod embedding problem (SEP) \cite{BCH2017} and characterized the geometric properties of the optimal stopping time $\tau$. As in the SEP, different cost functions lead to solutions given by the first hitting times of different Brownian functionals to corresponding boundaries.

More specifically, in this paper we focus on a special class of reward functions for which the solution to the optimal stopping problem is of the form
$$
\tau:=\inf\{t>0:B_t\le b(t)\}.
$$
The function $b$ can be associated with a barrier $\Rc:=\{(t,x):x \leq b(t)\}$. Throughout the paper, we call it the \textit{Root-type solution} to the distribution-constrained optimal stopping problem, since its barrier shares similarities with the Root barrier \cite{Root1969} for the SEP. Given a centred probability measure $\mu$ on $\R$ with finite first moment, the SEP consists of finding $\tau$ such that $B_{\tau} \sim \mu$ and $(B_{t \wedge \tau})_{t \geq 0}$ is uniformly integrable. The Root solution is associated with a barrier $\Rc^{SEP}$ such that if $(t,x) \in \Rc^{SEP}$, then $(s,x) \in \Rc^{SEP}$ for all $s>t$. For the Root-type solution to constrained stopping, if $(t,x) \in \Rc$, then $(t,y) \in \Rc$ for all $y<x$.

In \cite{BEES2018}, the authors characterized the Root-type solution as the optimizer of some cost functions. They also studied the maximum-type cost function $c=B_t^*$, where $B_t^*:=\max_{0\leq s\leq t}B_s$, and characterized the solution by finding $\tau\sim\mu$, with
$$
\tau:=\inf\{t\geq0:B_t-B_t^*\leq b(t)\}.
$$
The dependence of the solution on the cost function in distribution-constrained stopping is a feature shared with the optimal SEP: the Root solution \cite{Root1969}, the Rost solution \cite{Rost1976}, the Az{\'e}ma--Yor solution \cite{AY1979}, and the Perkins/Hobson--Pedersen solution \cite{P1985, HP2002} solve different optimal SEP problems.

The distribution-constrained optimal stopping problem has connections to many important and interesting problems in mathematical finance, probability, and stochastic control. When first introduced in \cite{BM2019}, Bayraktar and Miller applied the problem to study the model-independent super-replication of a volatility outlook. In \cite{BEES2018}, the problem is related to \textit{Shiryaev's problem}, or the inverse first-passage problem. Anulova \cite{Anulova1980} provided an early barrier-type solution to Shiryaev's problem. Motivated by the work of Avellaneda and Zhu \cite{AZ2001} on credit-risk modeling, Chen et al.~\cite{CCCS} gave a rigorous characterization of the barrier using a variational inequality. Ekstr\"om and Janson \cite{EJ} further provided an optimal-stopping representation and an integral-equation characterization.

More recently, the distribution-constrained stopping problem was related to the contract exit problem \cite{HTZ2023}. It can also be viewed as an optimal stopping adaptation of the mean-field planning problem of P. L. Lions; see \cite{Lions, RTTY}. In the spirit of \cite{XZ2013}, one can further relate it to some time-inconsistent stopping problems where the cost functions depend on the law of the stopping time.

Multiple stopping problems arise naturally in applications. For example, in \cite{BM2019}, extending the problem to multi-period trading requires multiple distributional constraints related to market-implied stock prices at multiple times. In \cite{HTZ2023}, if one considers the contract exit problem for multiple agents, it is also necessary to extend the formulation into a multiple-stopping context. Similar generalizations have already appeared in the SEP literature.

In this paper, we study a multiple-stopping extension of the distribution-constrained
stopping problem. Roughly speaking, we study
$$
\sup_{\tau_1\leq\ldots\leq\tau_n}
\E\left[c(B_{\tau_1},\ldots,B_{\tau_n},\tau_1,\ldots,\tau_n)\right],
\quad \text{subject to }\tau_i\sim\mu_i,\quad i=1,\ldots,n.
$$

The main difficulty with such a generalization is that the multiple stopping boundaries need not be \textit{ordered} by inclusion: for some $i<n$, the region associated with the $i$-th marginal may fail to contain that associated with the $(i+1)$-st. Consequently, the Brownian motion may enter a later stopping region before an earlier one. Hence the $n$-stopping problem cannot be obtained by simply combining $n$ one-stopping problems. In the context of the SEP, the difficulties of multiple stopping problems are well known, and solutions have been established in several settings: \cite{COT} for the multi-marginal Root solution and \cite{BHR, HLST2016, OS2017} for the multi-marginal Az{\'e}ma--Yor solution.
In the general setting, \cite{BCH2020} established the monotonicity principle for multi-marginal SEP problems by generalizing the concept of stop-go pairs through color swaps and multi-color swaps.

Our first main contribution is to provide a probabilistic characterization of the Root-type solution to the distribution-constrained multiple-stopping problem as a sequence of optimal stopping problems of an auxiliary backward process. This construction is in the same spirit as the multi-marginal Root construction of Cox, Ob{\l}{\'o}j and Touzi \cite{COT} in the SEP context; see also \cite{EJ}. As the multiple boundaries are not necessarily ordered, each cost function in the sequence of optimal stopping problems involves the value functions of the previous problems. We remark that similar techniques may also be useful for extending other types of distribution-constrained stopping to a multiple-stopping framework.

Our second main contribution is a direct martingale proof of the optimality of this construction. The argument is related to martingale inequality approach for optimal Skorokhod embeddings \cite{CW,COT,GHLT,HLST2016}, but the comparison is made in expectation sense rather than through a pathwise inequality.
Financially speaking, pathwise martingale inequalities for the SEP consist of the explicit construction of a semi-static superhedging strategy. In the distribution-constrained stopping context, optimality was previously proved using the monotonicity principle \cite{BEES2018}. Obtaining an analogous pathwise inequality in the present setting would require constructing explicit dual optimizers for the dual problem discussed in \cite[Remark~3.8]{BEES2018}.   In this work, we extend the optimality results in \cite{BEES2018} to the multiple stopping context and allow more general cost functions. Our proof compares after taking expectations and avoids the additional free-boundary regularity required by such a pathwise construction.

\textbf{Organization of the paper.} In Section~\ref{sec:main}, we formulate the multiple-stopping problem and state the main results. In Section~\ref{sec:proof_construction}, we prove the probabilistic characterization in terms of sequential optimal stopping for a backward process. In Section~\ref{sec:opti}, we prove the optimality of the construction using the martingale argument.

\subsection{Notations}

Throughout this paper, we work on a filtered probability space
$(\Omega,\Fc,(\Fc_t)_{t\ge0},\P)$ satisfying the usual hypotheses.
For $x\in\R$, we write $\P^x$ and $\E^x$ for probability and
expectation when the Brownian motion under consideration starts from
$x$. We take $\P=\P^0$ and write $\E=\E^0$.

Let $X=(X_s)_{s\ge0}$ be a standard Brownian motion with $X_0=0$. When a backward optimal stopping problem is considered with horizon $t\ge 0$ and state $x\in\R$, we use $Y=(Y_s)_{0\le s\le t}$ to denote a Brownian motion started from $x$, independent of $X$. The parameter $s$ represents elapsed time, while $t-s$ is the corresponding calendar time. Thus, $Y$ evolves forward as an ordinary Brownian motion; the term ``backward'' refers only to the decreasing time argument $t-s$. Whenever $X$ and $Y$ are used in the same argument, they are taken to be independent unless stated otherwise.

Let
$$
    \F^X=(\Fc_s^X)_{s\ge0}
    \quad\text{and}\quad
    \F^Y=(\Fc_s^Y)_{s\ge0}
$$
denote the natural filtrations generated by $X$ and $Y$, respectively,
augmented in the usual way. We write $\TT$ for the set of
$[0,\infty]$-valued $\F^X$-stopping times and, for $t\ge0$,
$\TT^t$ for the set of $\F^Y$-stopping times taking
values in $[0,t]$.

For a probability measure $\mu$ on $\R_+$, define its survival
function by
$$
    g^\mu(t):=\mu((t,\infty)),\qquad t\ge0.
$$
For probability measures $\mu$ and $\nu$ on $\R_+$, we define the stochastic order by
$$
    \mu\le_{\mathrm{st}}\nu
    \quad\Longleftrightarrow\quad
    g^\mu(t)\le g^\nu(t)
    \quad\text{for every }t\ge0.
$$

Finally, we use the conventions:
$$
    \inf\varnothing=+\infty,\qquad
    \sup\varnothing=-\infty,\qquad
    \sup\R=+\infty.
$$

\section{Distribution-constrained optimal multiple stopping: Root-type solution} \label{sec:main}

\subsection{Formulation of distribution-constrained optimal multiple stopping}

Throughout this paper, $\boldsymbol{\mu}_n:=(\mu_1,\ldots,\mu_n)$ represents a vector of $n$ probability measures on $\R_+$. For $k=1,\ldots,n$ and any $\mu_k$-integrable function $\phi:\R_+\to\R$, we write
$$
\mu_k(\phi):= \int_{\R_+} \phi(x) \mu_k(dx).
$$

Set $g^i:=g^{\mu_i}$. We are interested in the following $n$-marginal distribution-constrained stopping problem. Let $(\mu_i)_{1\leq i\leq n}$ be increasing in stochastic order, with each measure having a finite first moment, and define
$$\TT(\boldsymbol{\mu}_n):= \{(\tau_1,\ldots,\tau_n): \tau_i \in \TT, \, \tau_i \sim \mu_i, \, i=1,\ldots,n,\,\tau_1 \leq \cdots \leq \tau_n \, \P\text{-a.s.}\}.$$
The $n$-marginal distribution-constrained stopping problem is defined as:
$$
\mbox{ConsStop}(\boldsymbol{\mu}_n): \mbox{ find } \tau \in \TT(\boldsymbol{\mu}_n).
$$
Furthermore, we can consider the optimization problem:
$$
\sup_{\tau \in \TT(\boldsymbol{\mu}_n)} \E\left[ \sum_{i=1}^{n} c_i ( X_{\tau_i}, \tau_i)\right],
$$
for some measurable functions $(c_i)_{1 \leq i \leq n}$ such that the expectation is well-defined.

\begin{remark}

1. Here we focus on separable cost functions, as in \cite{COT, HLST2016}. In general, it is also possible to consider non-separable cost functions, but the associated optimality result will be more complicated.

2. The non-trivial part of the multiple-stopping problem is that the stopping boundaries in the $n$-stage problem need not be ordered; see also \cite{COT}. See Figure \ref{fig:non_order} for an illustration.

\end{remark}

\begin{figure}
    \centering
    \includegraphics[width=0.8\linewidth]{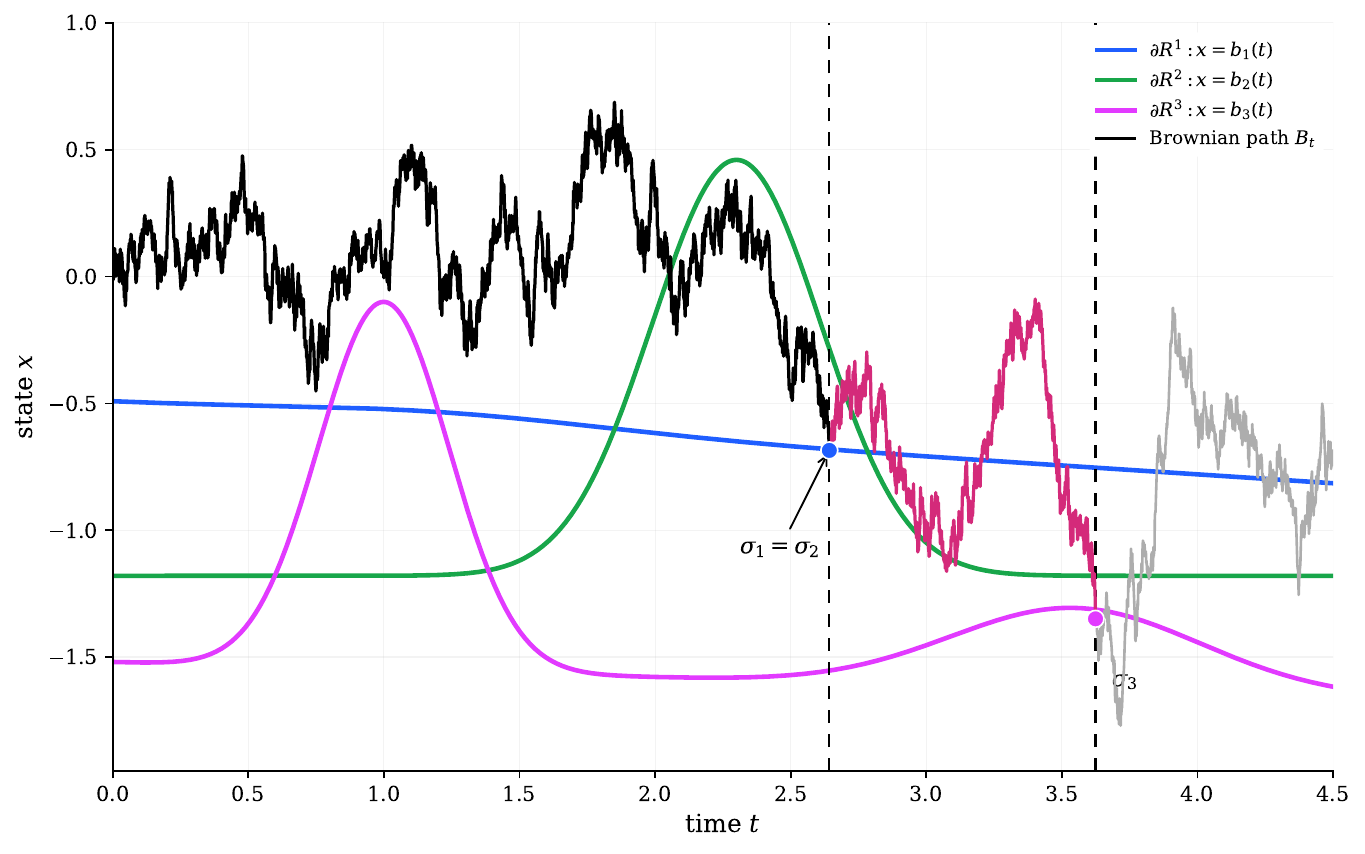}
    \caption{An example of three distribution-constrained stopping boundaries which are not ordered. In the realization above, the sample path enters $\Rc^2$ before $\Rc^1$, and $\Rc^3$ before $\Rc^2$, but these early entries do not trigger the corresponding stopping events. The correct stopping times should satisfy $\sigma_1=\sigma_2<\sigma_3$.}\label{fig:non_order}
\end{figure}

\begin{remark}[On the order condition]

The stochastic order seems most appropriate for the multi-marginal constrained stopping problem. On the one hand, since $\tau_1 \leq \ldots \leq \tau_n$, the measures in $\boldsymbol{\mu}_n$ must be in stochastic order. On the other hand, the stochastic order does not force the boundaries to be ordered. This is precisely why the problem need not reduce to separate one-stopping problems. See Example \ref{exam:sto_order}.

We can also consider imposing stronger orders. For example, the hazard-rate order is defined by comparing the hazard-rate functions: for an absolutely continuous law, the hazard rate is defined as $h_{\mu}(t):=\frac{f_{\mu}(t)}{g^{\mu}(t)}$ on $\{g^\mu>0\}$, where $f_{\mu}$ and $g^{\mu}$ are, respectively, the p.d.f. and survival function of $\mu$. Writing $\le_{\mathrm{hr}}$ for the hazard-rate order, we have
$$
\mu_1 \le_{\mathrm{hr}} \mu_2 \Longleftrightarrow h_{\mu_1} \geq h_{\mu_2}.
$$
It is well known that the hazard-rate order is stronger than the stochastic order: $\mu_1\le_{\mathrm{hr}}\mu_2\Rightarrow\mu_1\le_{\mathrm{st}}\mu_2$. By adapting the arguments of \cite{KK2022}, one can show that if $\mu_1\le_{\mathrm{hr}}\mu_2$, then the multiple boundaries are in order. In this case, the multiple-stopping problem can be reduced to separate one-stopping problems and is therefore not of interest here. The role of the hazard-rate order in distribution-constrained multiple stopping is similar to that of the mean residual order in the multi-marginal Az{\'e}ma--Yor solution to the SEP; see \cite{MadanYor}.

\end{remark}

\begin{example}[Stochastic order can result in non-ordered boundaries]
 \label{exam:sto_order}
    Let $X$ be a standard Brownian motion and let $\Phi$ denote the standard normal distribution function. For $t>0$, define
    $g^1(t):=2\Phi\left(\frac{1}{\sqrt t}\right)-1$, which is the survival function of the first hitting time of the horizontal line $x=-1$.
    We further define
    \begin{align*}
        g^2(t) :=
        \begin{cases}
            g^1(t), & 0<t\le 1,\\
            g^1(t)+\displaystyle\int_1^t g^1(t-s)\,\mathrm d(1-g^1)(s), & t>1.
        \end{cases}
    \end{align*}
    Let $\beta_1,\beta_2$ be the probability laws with survival functions $g^1,g^2$, respectively. It is clear that
    $g^1(t)\le g^2(t)$, and consequently $\beta_1 \le_{\mathrm{st}} \beta_2$.

    We can now consider the following two stopping regions:
    $$\Rc^1 := \{(t,x) \in [0, \infty) \times \R: x \leq -1\} \quad\text{and} \quad \Rc^2 := \left([0,1]\times \R\right)\,\cup\,\left((1,\infty)\times(-\infty,-2]\right).$$
    Define the stopping times associated with these two regions by
    $$\sigma_1 := \inf\{t > 0: (t, X_t) \in \Rc^1\} \quad\text{and} \quad \sigma_2 := \inf\{t\ge \sigma_1:(t,X_t)\in \Rc^2\}.$$
    It is not difficult to verify that $\sigma_1 \sim \beta_1$ and $\sigma_2 \sim \beta_2$. However, it is clear that the two stopping regions are not ordered. See Figure \ref{fig:sto_order} for more details.

\end{example}

\begin{figure}
    \centering
    \includegraphics[width=0.8\linewidth]{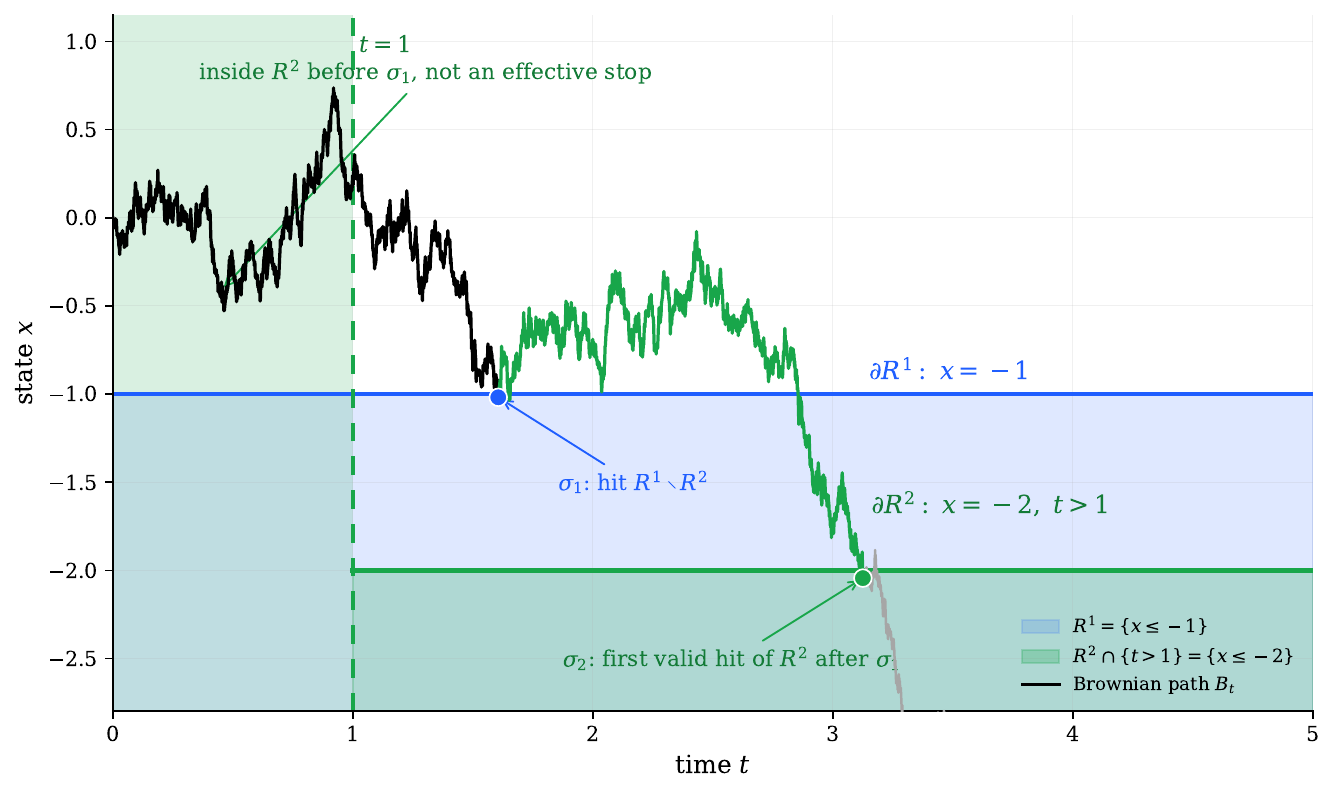}
    \caption{A realization of the non-ordered stopping regions in Example~\ref{exam:sto_order}. The Brownian path may enter $\Rc^2$ before $\Rc^1$, but this entry does not provide an effective stopping time.}\label{fig:sto_order}
\end{figure}

\subsection{Probabilistic characterization of the multiple-stopping boundary}\label{sec:proba}

In this section, we provide a probabilistic characterization of $\mbox{ConsStop}(\boldsymbol{\mu}_n)$. For this purpose, we first define a sequence of barriers and the corresponding stopping times. The construction is in the same spirit as \cite{COT}, but needs to be adapted to the current context.

We start by defining the first barrier. Consider an optimal stopping problem with respect to the backward process $Y$, and denote its value function by $u^1$.
$$u^1(t,x):= \inf_{\tau \in \TT^t}\E^x\left[\mathbf{1}_{\{Y_\tau < 0\}} + \left(g^1(t-\tau) - \mathbf{1}_{\{Y_\tau < 0\}}\right)\mathbf{1}_{\{\tau < t\}}\right].$$
We further define the corresponding stopping region by
$$\Rc^1 := \{(t,x) \in (0,\infty) \times \R: u^1(t,x) = g^1(t)\},$$
and define the first stopping time by $\sigma_1 := \inf\{t > 0: (t, X_t) \in \Rc^1\}$.

For $t>0$, the symmetry of Brownian motion and the change $x \mapsto -x$ identify $u^1(t,x)$ with the value function in \cite[equation~(4)]{EJ}. The only discrepancy is the terminal convention $Y_t < 0$ here versus $Y_t \leq 0$ after reflection, which does not affect the value. Under the same reflection, the barrier constructed in \cite[Theorem~2.3]{EJ} is precisely $\Rc^1$. Hence $\Rc^1$ is a closed lower barrier and $\sigma_1\sim\mu_1$, while \cite[Theorem~2.6]{EJ} gives the value identity for $t>0$. At $t=0$, the value identity follows directly from $X_0=0$ and $\P(\sigma_1=0)=\mu_1(\{0\})=0$. Consequently,
$$
\sigma_1\sim\mu_1,
\qquad
u^1(t,x)=\P(\sigma_1>t, X_t>x),
\qquad (t,x)\in[0,\infty)\times\R.
$$

For $k\geq2$, set $w^k(t):=g^k(t)-g^{k-1}(t)$ and define
$$u^k(t,x):=\inf_{\tau\in\TT^t}\E^x\left[u^{k-1}(t-\tau,Y_\tau)+w^k(t-\tau)\mathbf{1}_{\{\tau<t\}}\right],\quad t\geq0,\quad x\in\R.$$
The corresponding stopping region and stopping time are naturally defined by
$$\Rc^k:=\{(t,x)\in[0,\infty)\times\R:u^k(t,x)=u^{k-1}(t,x)+w^k(t)\},$$
and
$$
    \sigma_k := \inf\left\{t \ge \sigma_{k-1}: (t, X_t) \in \Rc^k\right\}, \quad k=2,\ldots,n.
$$
We now state our first main result, a probabilistic characterization of the distribution-constrained multiple stopping problem.

\begin{theorem}
    \label{thm::goal}
    Let $\mu_1, \ldots , \mu_n$ be atomless probability measures on $\R_+$, each with finite first moment, such that $\mu_1 \le_{\mathrm{st}} \cdots \le_{\mathrm{st}} \mu_n$.
    Then
$$
(\sigma_1,\ldots,\sigma_n)\, \quad\mbox{solves}\, \quad\mbox{ConsStop}(\boldsymbol{\mu}_n).
$$
Moreover, for $k=1,\ldots,n$,
$$
u^k(t,x)=\P(\sigma_k>t,\;X_t>x),
\qquad (t,x)\in [0,\infty)\times\R.
$$
\end{theorem}

\textbf{An induction approach}

\vspace{3mm}

As in \cite{COT}, the rigorous proof of the above theorem is based on an induction. The output of each induction step is a stopping time $\sigma^{\xi}$ with law $\xi$, which then serves as the starting point for the subsequent definitions. Given a new law $\beta$, we construct a new stopping time $\sigma^{\beta}$. Repeating this procedure proves the main theorem.

Let $\xi \le_{\mathrm{st}}\beta$ and $w^\beta(t):= g^\beta(t) - g^\xi(t)$. Define
$$v^\xi(t,x) := \P\left(\sigma^\xi > t, X_t >x\right).$$
The optimal stopping problem of interest in our induction procedure is
$$u^\beta(t,x) := \inf_{\tau \in \TT^t}\E^x\left[v^\xi(t-\tau,Y_\tau) + w^\beta(t-\tau)\mathbf{1}_{\{\tau < t\}}\right].$$
We introduce the corresponding stopping region
$$\Rc^\beta:= \left\{(t,x): u^\beta(t,x) = v^\xi(t,x) + w^\beta(t)\right\},$$
and define
$$\sigma^{\beta}:= \inf\left\{t \geq \sigma^\xi: (t, X_t) \in \Rc^\beta\right\}, \quad v^\beta(t,x) := \P\left(\sigma^\beta > t, X_t > x\right).$$
The following theorem is the main result obtained from the induction procedure.
\begin{theorem}
\label{thm::main}
    Let $\xi$ and $\beta$ be atomless probability measures on $\R_+$, each with a finite first moment, and assume that $\xi\le_{\mathrm{st}}\beta$. Suppose further that, for some $k\geq1$, $\sigma^\xi=\sigma_k$ is the $k$-th stopping time generated by the preceding recursion and that \(\sigma^\xi\sim\xi\). Then $\Rc^\beta$ is a closed lower barrier. Moreover, $u^\beta=v^\beta$ and $\sigma^\beta\sim\beta$.
\end{theorem}

\begin{remark}
The theorem is similar in spirit to \cite[Theorem~4.1]{COT} in its recursive
form, but the two constructions use different functions. In
\cite[equation~(3.4)]{COT}, \(u^k\) is a potential determined by the
stopped process \(X_{t\wedge\sigma_k}\). Here, $v^\xi(t,x)$
is the probability that stopping has not occurred by time \(t\) and
that \(X_t>x\). Moreover, the increment \(g^k-g^{k-1}\) is determined
by the laws of the stopping times. Consequently, the proof
does not follow directly from \cite{COT} and requires different
arguments.
\end{remark}

We next show that Theorem~\ref{thm::goal} follows from Theorem~\ref{thm::main}.

\begin{proof}[Proof of Theorem~\ref{thm::goal}]
    The one-marginal conclusion established above proves the assertion for $k=1$.

    Suppose that the assertion holds through stage $k-1$, where $2 \leq k \leq n$, and set
    $$\xi:= \mu_{k-1}, \quad \beta:= \mu_k, \quad \sigma^\xi := \sigma_{k-1}.$$
    Then $\sigma^\xi$ is the $(k-1)$-st stopping time generated by the recursion, $\sigma^\xi \sim \xi$, $v^\xi = u^{k-1}$, and $w^\beta = w^k$.
    Theorem~\ref{thm::main} therefore gives
    $$\sigma_k \sim \mu_k, \quad u^k(t,x) = \P(\sigma_k>t, X_t > x),$$
    and also shows that $\Rc^k$ is a closed lower barrier. This closes the induction.
\end{proof}

The proof of Theorem~\ref{thm::main} is the focus of Section~\ref{sec:proof_construction}.

\subsection{Optimality via martingale arguments} \label{sec:optimal}

In this section, we prove the optimality of the previously constructed solution to the $n$-marginal distribution-constrained stopping problem. Whereas the corresponding one-marginal results in
\cite{BEES2018} were obtained from the monotonicity principle, our proof uses a martingale argument at the expectation level.

Recall that the stopping times $(\sigma_i)_{1 \leq i \leq n}$ are defined in Section \ref{sec:proba}. For \(F\in C^{1,2}(\R_+\times\R)\), we introduce the operator $\mathcal G$ by
$$
    \mathcal G F(t,x)
    :=
    \left(\partial_t+\frac12\partial_{xx}\right)F(t,x).
$$
\begin{theorem}
\label{thm:opt-multi} Assume the hypotheses of Theorem~\ref{thm::goal}, and let $\sigma = (\sigma_1,\ldots,\sigma_n)$ be the stopping times constructed there.
For $i=1,\ldots,n$, let $F_i\in C^{1,2}(\R_+\times\R)$. For each $T<\infty$, suppose that
$$
    |F_i(t,x)|+|\mathcal G F_i(t,x)|
    \le
    K_{i,T}(1+|x|^{r_{i,T}}),
    \qquad
    0\le t\le T,\quad x\in\R,
$$
for some $K_{i,T}<\infty$ and $r_{i,T}\ge0$, and assume that $x\longmapsto\mathcal G F_i(t,x)$
is non-decreasing for almost every $t\ge0$. Let $\rho\in\TT(\boldsymbol{\mu}_n)$. Suppose that for all
$i=1,\ldots,n$ and $\theta\in\{\rho_i,\sigma_i\}$, the family
   $ \left\{
        F_i(t\wedge\theta,X_{t\wedge\theta})
        : t\ge0
    \right\}$
is uniformly integrable. Then
$$
    \E[F_i(\sigma_i,X_{\sigma_i})]
    \ge
    \E[F_i(\rho_i,X_{\rho_i})],
    \qquad i=1,\ldots,n.
$$
If the uniform-integrability condition holds for every
$\rho\in\TT(\boldsymbol{\mu}_n)$, then $\sigma$ maximizes
$$
    \rho
    \longmapsto
    \E\!\left[
        \sum_{i=1}^n F_i(\rho_i,X_{\rho_i})
    \right]
$$
over $\TT(\boldsymbol{\mu}_n)$.
\end{theorem}

The proof of the above theorem is postponed to Section \ref{sec:opti}. Throughout the following three corollaries, assume the hypotheses of Theorem~\ref{thm::goal}.

\begin{corollary}\label{opt:case1}
Let \(p\ge2\). For \(i=1,\ldots,n\), let $F_i\in C^2(\R)$ have a non-decreasing second order derivative and satisfy $|F_i(x)|\le K_i(1+|x|^p)$ for some constant $K_i$. Assume further that $\int_0^\infty t^{p/2}\,\mu_n(\mathrm dt)<\infty$.
Then, for every $\rho\in\TT(\boldsymbol{\mu}_n)$ and every
$i=1,\ldots,n$,
\begin{equation}\label{eq:state-stagewise}
    \E[F_i(X_{\sigma_i})]
    \ge
    \E[F_i(X_{\rho_i})].
\end{equation}
In particular,
$$
    \E\!\left[\sum_{i=1}^n F_i(X_{\sigma_i})\right]
    \ge
    \E\!\left[\sum_{i=1}^n F_i(X_{\rho_i})\right].
$$
\end{corollary}

\begin{corollary}\label{opt:case2}
Let $q\ge0$, and let $A\in C^1(\R_+)$ be non-decreasing and satisfy $|A(t)|\le K(1+t^q)$
for some $K>0$. Suppose in addition that $\int_0^\infty t^{q+1/2}\,\mu_n(\mathrm dt)<\infty$. Then, for every $\rho\in\TT(\boldsymbol{\mu}_n)$ and every
$i=1,\ldots,n$,
$$
    \E[A(\sigma_i)X_{\sigma_i}]
    \ge
    \E[A(\rho_i)X_{\rho_i}].
$$
\end{corollary}

\begin{corollary}\label{opt:case3}
Let $q\ge0$ and $p\ge2$, and let
$A\in C^1(\R_+)$ be nonnegative and non-decreasing, with $|A(t)|\le K(1+t^q)$
for some constant $K>0$. Let $F\in C^2(\R)$ be non-decreasing, with $F''$ non-decreasing, and suppose that
    $|F(x)|\le K_F(1+|x|^p)$ for some constant $K_F$.
Suppose in addition that $\int_0^\infty t^{q+p/2}\,\mu_n(\mathrm dt)<\infty$.
Then, for every $\rho\in\TT(\boldsymbol{\mu}_n)$ and every
$i=1,\ldots,n$,
$$
    \E[A(\sigma_i)F(X_{\sigma_i})]
    \ge
    \E[A(\rho_i)F(X_{\rho_i})].
$$
\end{corollary}

\begin{remark}

For atomless target laws, Corollaries~\ref{opt:case1}
and~\ref{opt:case2} with \(n=1\) yield the two types of maximizing
inequalities appearing in Corollary~1.1 of \cite{BEES2018}, for payoffs
satisfying the present hypotheses. For the state payoff, monotonicity
of \(F_i''\) replaces strict positivity of the third derivative, and
the \(p/2\)-moment assumption is sufficient. For the time--state payoff,
we impose \(A\in C^1\), while allowing non-strict monotonicity and the \((q+1/2)\)-moment assumption.
\end{remark}

\begin{remark}
The proof of Theorem~\ref{thm:opt-multi} first establishes a comparison that is independent of the payoff and then combines it with It\^o's formula. The classical Root embedding has a similar structure, although its terminal-state constraint leads to a different comparison.

Let \(\tau_R\) be the Root embedding of a centred law \(\nu\) with finite second moment, and let \(\theta\) be any stopping time with finite first moment such that \(X_\theta\sim\nu\). Define
$$
P^\theta(t,x):=-\E\bigl[|X_{t\wedge\theta}-x|\bigr].
$$
Root's stopped-potential order is
$$
P^{\tau_R}(t,x)\le P^\theta(t,x),
\qquad t\ge0,\quad x\in\mathbb R.
$$
This comparison implies
$$
\E[(\tau_R-t)^+]
\le
\E[(\theta-t)^+],
\quad t\ge0,
$$
which yields the classical convex-time optimality of Root's embedding; see \cite{Rost1976,CW}.

In the present problem, the proof of Theorem~\ref{thm:opt-multi} is based on the analogous comparison \eqref{eq:sur-tail}:
$$
\P(\rho_i>t,X_t>x)
\le
\P(\sigma_i>t,X_t>x).
$$
Thus, the terminal-state constraint in the SEP leads to an order of stopped potentials, whereas the time-law constraint in the present problem leads to the comparison in \eqref{eq:sur-tail}. In both cases, the payoff inequalities follow from the corresponding comparison.
\end{remark}

\begin{remark}
We make a comparison between our approach and the classical pathwise martingale inequality arguments, by first recalling the latter in the context of
Root embedding. Let \(\varphi:\R_+\to\R\) be convex and increasing, with \(\varphi(0)=0\), and let \(f=\varphi'_+\). In
\cite[Section~5]{CW}, functions \(G\) and \(H\), with \(H\) depending
only on the state, are constructed so that
\[
G(x,t)+H(x)\le \varphi(t),
\]
with equality on the Root barrier \cite[Proposition~5.1]{CW}. The
construction integrates in the time variable and uses the fact that
the integrand agrees with \(f\) throughout the stopping region. It
therefore does not require differentiating the free boundary. Under
the assumptions imposed there, $G(X_t,t)$ is a submartingale and its
Root-stopped version is a martingale \cite[Lemma~5.2]{CW}.

Motivated by the dual formulation in
\cite[Remark~3.8]{BEES2018}, with the inequality reversed for our
maximization problem, a pathwise proof for the $i$-th payoff in our setting would seek
$$
F_i(t,X_t)\le \psi_i(t)+N^i_t,
$$
where $\psi_i$ depends only on time and $N^i$ is a true martingale
satisfying $\E[N^i_{\rho_i}]=N^i_0$ for every admissible $\r\in\TT(\boldsymbol{\mu}_n)$, with
equality upon stopping at $\sigma_i$. Since $\r_i\sim\mu_i$,
$$
\E[\psi_i(\r_i)]
   =\int_{\R_+}\psi_i(t)\,\mu_i(\mathrm dt),
$$
so the expectation of the static term is fixed by the prescribed
time law.

By contrast, the static terms in the multi-marginal Root inequality
of \cite[Appendix~A]{COT} are functions $\lambda_i$ of the stopped states, and their expectations are fixed by the prescribed state marginals. Reversing the roles of time and space in this Root construction would instead lead to a spatial integral with endpoint $X_t$. Its decomposition across $x=b(t)$ would require sufficient regularity for a change-of-variable formula with local time and a condition, such as smooth fit, controlling the local-time term \cite{P2005}. Our It\^o argument avoids these requirements by applying
the formula directly to the given $C^{1,2}$ payoff, without differentiating the free boundary.
\end{remark}

\begin{remark}
\label{rem:ou-extension}
In this paper, the proved results are restricted to Brownian motion. However, as in \cite{COT}, our methods may extend to more general processes. The following calculation suggests a possible extension to Ornstein--Uhlenbeck dynamics. Consider
\[
    dX_t=\kappa(m-X_t)\,dt+\nu\,dW_t,
    \qquad
    d\widehat Y_t=\kappa(\widehat Y_t-m)\,dt+\nu\,d\widehat W_t,
    \qquad m\in\mathbb R,\quad \kappa,\nu>0,
\]
where $W$ and $\widehat W$ are standard Brownian motions.
For $X_0=z$ and $\widehat Y_0=x$, the explicit solutions provide
\[
X_t^z\sim
\Nc\left(
m+e^{-\kappa t}(z-m),
\frac{\nu^2}{2\kappa}(1-e^{-2\kappa t})
\right), \
\widehat Y_t^x\sim
\Nc\left(
m+e^{\kappa t}(x-m),
\frac{\nu^2}{2\kappa}(e^{2\kappa t}-1)
\right).
\]
Consequently,
\[
z-\widehat Y_t^x
   \overset{d}{=}
e^{\kappa t}(X_t^z-x),
\]
and, since $e^{\kappa t}>0$, we have
\[
\mathbb P(\widehat Y_t^x<z)
   =\mathbb P(X_t^z>x),
\]
which is the OU analogue of \eqref{eq:switching-tilde-ind}. Together
with the strong Markov property, the preceding probability identity
suggests that Lemma~\ref{lem::V-submartingale} also admits an OU
counterpart. After defining \(u^\beta\) using \(\widehat Y\) and
defining \(v^{\beta}\) using the OU process \(X\), the Brownian proof of \(u^\beta=v^{\beta}\) is expected to extend to the current context.

Assuming that this identity holds at every stage, a backward-sampling
argument will lead to, for every admissible competitor $\rho$,
\[
    \E\!\left[
        h(X_t)\mathbf 1_{\{\rho_i>t\}}
    \right]
    \le
    \E\!\left[
        h(X_t)\mathbf 1_{\{\sigma_i>t\}}
    \right]
\]
for every non-decreasing Borel function \(h\) for which both
expectations are finite.

Finally, for a sufficiently regular payoff \(F_i\), the relevant drift
operator is
\[
    \mathcal G F_i(t,x)
      :=\partial_t F_i(t,x)
        +\kappa(m-x)\partial_x F_i(t,x)
        +\frac{\nu^2}{2}\partial_{xx}F_i(t,x).
\]
Under the analogous growth and uniform-integrability conditions, the
It\^o argument used above would yield OU optimality whenever
\[
    x\longmapsto\mathcal G F_i(t,x)
\]
is non-decreasing for almost every \(t\). These observations describe
a plausible extension of the present framework to Ornstein--Uhlenbeck
dynamics.
\end{remark}

\section{Proof of Theorem \ref{thm::main}} \label{sec:proof_construction}

In this section, we will prove the main results. We first establish several useful lemmas. Then, in Section \ref{sec:finite_support}, we prove the result under an additional finite-support assumption. Section~\ref{sec:general_case} then removes this assumption and
completes the proof of Theorem~\ref{thm::main}.

\subsection{Useful lemmas}

We first present a switching identity that is useful for the proof of the submartingale property.

\begin{lemma}
\label{lem::BM-switch}
    Let $X$ and $Y$ be independent standard Brownian motions.
For $r>0$, $a,b\in\R$, and any bounded Borel $\varphi:\R\to\R$,
\begin{equation}\label{eq:switching-tilde-bdd}
\E_Y^{b}\left[\varphi\left(a-Y_r\right)\right]
=\E_X^{a}\left[\varphi\left(X_r-b\right)\right].
\end{equation}
In particular, taking $\varphi(z)={\bf 1}_{\{z>0\}}$ yields
\begin{equation}\label{eq:switching-tilde-ind}
\P_Y^{b}(a>Y_r)=\P_X^{a}(X_r>b).
\end{equation}
\end{lemma}

\begin{proof}
Write $Y_r=b+Z$ and $X_r=a+Z'$, where $Z,Z'\sim N(0,r)$.
Since $Z'\stackrel{d}{=}-Z$, we have
\[
\E_Y^{\,b}\left[\varphi(a-Y_r)\right]
=\E\big[\varphi(a-b-Z)\big]
=\E\big[\varphi(a-b+Z')\big]
=\E_X^{a}\!\left[\varphi(X_r-b)\right],
\]
which proves \eqref{eq:switching-tilde-bdd}. The indicator case
\eqref{eq:switching-tilde-ind} corresponds to $\varphi(z)={\bf 1}_{\{z>0\}}$.
\end{proof}

\begin{lemma}\label{lem::V-submartingale}
Let $\sigma^\xi\in\TT$ have law $\xi$. For $t\ge0$ and $x\in\R$,
define
\[
V^t_s := v^\xi(t-s,Y_s),\qquad 0\le s\le t.
\]
Then $(V^t_s)_{0\le s\le t}$ is a $\mathbb{P}^x$-submartingale with respect to
the filtration $(\mathcal{F}^Y_s)_{0\le s\le t}$ generated by $Y$.
\end{lemma}

\begin{proof}

Let $0\le u\le s\le t$. Using the definition of $v^\xi$ and independence of
$X$ and $Y$, we have
\[
\E^x_Y\left[V_s^t\big|\mathcal F_u^Y\right]
=\E^x_Y\left[\P\left[\sigma^\xi> t-s, X_{t-s}>Y_s\right]\big|\Fc_u^Y\right]
=\E^x_Y\left[\E_X\left[\mathbf{1}_{\{\sigma^\xi> t-s\}}\mathbf{1}_{\{X_{t-s}>Y_s\}}\right]\big|\Fc_u^Y\right].
\]
Define $h(y):= \E_X\left[\mathbf{1}_{\{\sigma^\xi> t-s\}}\mathbf{1}_{\{X_{t-s}>y\}}\right]$. Then $h$ is non-increasing in $y$, hence Borel measurable and bounded. In particular, $h(Y_s)$ is $\Fc_s^Y$-measurable. Therefore, by the strong Markov property of $Y$,
$$\E_Y^x\left[h(Y_{s}) \mid \Fc_u^Y\right] = \tilde{\E}^{Y_u}_{\tilde{Y}}[h(\tilde{Y}_{s-u})],$$
where $\tilde{Y}$ is an independent copy of $Y$ started from $\tilde{Y}_0 = Y_u$, and $\tilde{\E}$ is the corresponding expectation with respect to $\tilde{Y}$.

Also introduce $\tilde{X}$ as an independent copy of $X$. Then, by Fubini's theorem, Lemma~\ref{lem::BM-switch}, and the strong Markov property, one can compute
    \begin{align*}
    \E^x_Y\left[\E_X\left[\mathbf{1}_{\{\sigma^\xi> t-s\}}\mathbf{1}_{\{X_{t-s}>Y_s\}}\right]\big|\Fc_u^Y\right]
    &=\E_X\left[\tilde{\E}_{\tilde{Y}}^{Y_u}\left[\mathbf{1}_{\{\sigma^\xi > t-s}\}\mathbf{1}_{\{X_{t-s}> \tilde{Y}_{s-u}\}}\right]\right] \\
    &=\E_X\left[\mathbf{1}_{\{\sigma^\xi > t-s\}}\tilde{\E}_{\tilde X}^{X_{t-s}}\left[\mathbf{1}_{\{\tilde{X}_{s-u} > Y_u\}}\right]\right] \\
    &=\E_X\left[\mathbf{1}_{\{\sigma^\xi > t-s\}}\mathbf{1}_{\{X_{t-u} > Y_u\}}\right] \\
    &\geq \E_X\left[\mathbf{1}_{\{\sigma^\xi > t-u\}}\mathbf{1}_{\{X_{t-u}> Y_u\}}\right] \\
    &=v^\xi(t-u, Y_u) = V_u^t,
\end{align*}
where the last inequality follows because $t-u\geq t-s$ implies $\{\sigma^\xi>t-u\}\subseteq\{\sigma^\xi>t-s\}$. This completes the proof.
\end{proof}

\begin{corollary}\label{cor::subm-lower-tilde}
For $r \ge 0$ and $x \in \R$, define
$$
\overline{v}^\xi(r,x):=\P\left[\sigma^\xi > r,\ X_r\le x\right],\qquad 0\le r<\infty.
$$
For $0\le s\le t$, let $\overline{V}_s^t:=\overline{v}^\xi(t-s,Y_s)$.
Then $(\overline V_s^t)_{0\le s\le t}$ is a $\P^x$-submartingale with respect to
$(\mathcal F_s^Y)_{0\le s\le t}$.
\end{corollary}

\begin{proof}
Repeat the proof of Lemma~\ref{lem::V-submartingale}, replacing the indicator
${\bf 1}_{\{X_{t-s}>Y_s\}}$ by ${\bf 1}_{\{X_{t-s}\le Y_s\}}$ throughout.
The only change in the switching step is to apply Lemma~\ref{lem::BM-switch}
with $\varphi(z)={\bf 1}_{\{z\le0\}}$, which is again bounded Borel.
\end{proof}

The following lemma is useful for determining whether $(t,x) \in \Rc^\beta$.

\begin{lemma}
    \label{lem::uvw}
    $v^\xi(t,x) \leq u^\beta(t,x) \leq v^\xi(t,x) + w^\beta(t)$ for $t \geq 0$ and $x \in \R$. Consequently, if $w^\beta(t) = 0$, then $(t,x) \in \Rc^\beta$.
\end{lemma}

\begin{proof}
    Recall that
    $$u^\beta(t,x) = \inf_{\tau \in \TT^t}\E^x\left[v^\xi(t-\tau, Y_\tau) + w^\beta(t-\tau)\mathbf{1}_{\{\tau < t\}}\right].$$
    Choosing $\tau = 0$, we have
    $u^\beta(t,x) \leq v^\xi(t,x) + w^\beta(t)$.
    Moreover, since $w^\beta(t) \geq 0$ and $V^t_s := v^\xi(t-s, Y_s)$ is a submartingale,
    \begin{align*}
        u^\beta(t,x) \geq \inf_{\tau \in \TT^t} \E^x[v^\xi(t-\tau, Y_\tau)] = v^\xi(t,x).
    \end{align*}
    Hence, if $w^\beta(t) = 0$, then $u^\beta(t,x) = v^\xi(t,x)$ by the previous argument. Recall that $\Rc^\beta := \{(t,x): u^\beta(t,x) = v^\xi(t,x) + w^\beta(t)\}$. Thus $(t,x) \in \Rc^\beta$.
\end{proof}

\subsection{Proof of Theorem~\ref{thm::main} under an additional finite-support assumption}
\label{sec:finite_support}

Let $\sigma^\xi$ be an $(\mathcal F_t^X)$-stopping time in the induction procedure of Section~\ref{sec:proba}.
Assume that $\xi$ is atomless, $\xi\le_{\mathrm{st}}\beta$, $\sigma^\xi\sim\xi$, and $\sigma^\xi>0$ almost surely.
If \(\beta=\xi\), then \(w^\beta\equiv0\), hence
Lemma~\ref{lem::uvw}, together with the definition of $\Rc^\beta$, gives \(u^\beta=v^\xi\). Therefore, we
assume below $\beta \neq \xi$. Let
$$
    \ell_\xi^\beta
    :=
    \inf\{t:\beta((0,t))\neq\xi((0,t))\},
    \qquad
    r_\xi^\beta
    :=
    \sup\{t:\beta((t,\infty))\neq\xi((t,\infty))\},
$$
and set
\[
    I_\xi^\beta:=[\ell_\xi^\beta,r_\xi^\beta),
    \qquad
    \Lambda
    :=
    \{\ell_\xi^\beta\}
    \cup
    \{s\in I_\xi^\beta:\beta(\{s\})>0\}.
\]

In this section, we consider the case where $\beta|_{I_\xi^\beta}$ is finitely supported.
\begin{assumption} \label{assu:finite_supp}
    Suppose $0\le \ell_\xi^\beta<r_\xi^\beta<\infty$, $g^\beta(0)=1$,
and that \(\beta|_{I_\xi^\beta}\) is finitely supported. More precisely, $\Lambda$ can be written as
$
    \Lambda=\{t_0,t_1,\ldots,t_m\},
$
with
$t_0=\ell_\xi^\beta<t_1<\ldots<t_m
    <r_\xi^\beta=t_{m+1}$.
\end{assumption}
The two laws agree before $t_0$, hence $w^\beta(t)=0$ for \(t<t_0\).
As \(\xi\) is atomless, $w^\beta(t_0)=-\beta(\{t_0\})$.
By \(w^\beta\ge0\), it follows that
$$
    \beta(\{t_0\})=w^\beta(t_0)=0,
    \qquad
    w^\beta(t)=0,\quad 0\le t\le t_0.
$$
By the definition of \(r_\xi^\beta\), \(w^\beta(t)=0\) for
\(t>r_\xi^\beta\). Consequently, right-continuity gives $w^\beta(t)=0$ for every $t\ge r_\xi^\beta$.
The main result of this section is the following.
\begin{proposition}
\label{prop:tc-grid-value-contact}
Suppose that Assumption~\ref{assu:finite_supp} holds. Then, for every $t\ge0$ and $x\in\mathbb R$,
\begin{equation}
\label{eq:tc-grid-value-direct}
    u^\beta(t,x)
    =
    \P(
        \sigma^{\beta}>t,\,
        X_t>x
    ).
\end{equation}
Consequently,
\[
    u^\beta=v^{\beta},
    \qquad
    \sigma^{\beta}\sim\beta.
\]
\end{proposition}

Below, we decompose the proof of Proposition~\ref{prop:tc-grid-value-contact} into several steps.

\subsubsection{Reduction to finitely valued stopping times}

Starting from a horizon $t$, a stopping time $\rho$ corresponds to the calendar time $t-\rho$. For
$t\in[t_0,r_\xi^\beta)$, introduce
$$
    k(t):=\max\{0\le i\le m:t_i\le t\}, \ \ \mathcal D_t :=\{t-t_i:0\le i\le k(t)\},
$$
and $\TT_t^{\Lambda}
    :=
    \{\rho\in\TT^t:\rho\text{ takes values in }\mathcal D_t\}$.

\begin{lemma}
\label{lem:tc-grid-reduction}
For every \(t\in[t_0,r_\xi^\beta)\) and \(x\in\mathbb R\),
\begin{equation}
\label{eq:tc-exact-grid-reduction}
    u^\beta(t,x)
    =
    \inf_{\rho\in\TT_t^{\Lambda}}
    \E^x\!\left[
        v^\xi(t-\rho,Y_\rho)
        +
        w^\beta(t-\rho)\mathbf 1_{\{\rho<t\}}
    \right].
\end{equation}
\end{lemma}

\begin{proof}
For $\rho\in\TT^t$, let
$
J_{t,x}(\rho)
:=
\E^x\!\left[
v^\xi(t-\rho,Y_\rho)
+w^\beta(t-\rho)\mathbf 1_{\{\rho<t\}}
\right]$.
Since $\TT_t^{\Lambda}\subseteq\TT^t$,
\[
u^\beta(t,x)
=\inf_{\rho\in\TT^t}J_{t,x}(\rho)
\le
\inf_{\rho\in\TT_t^{\Lambda}}J_{t,x}(\rho).
\]
On the other hand, by Lemma~\ref{lem:tc-grid-rounding}, for every
$\tau\in\TT^t$, there exists
$\bar\tau\in\TT_t^{\Lambda}$ such that
$J_{t,x}(\bar\tau)\le J_{t,x}(\tau)$.
Consequently,
\[
\inf_{\rho\in\TT_t^{\Lambda}}J_{t,x}(\rho)
\le J_{t,x}(\bar\tau)
\le J_{t,x}(\tau).
\]
Taking the infimum over $\tau\in\TT^t$, we obtain the reverse
inequality and hence the equality.
\end{proof}

\subsubsection{Study of the reduced stopping problem}


The next proposition gives the recursion from $t_{i-1}$ to $t_i$ and identifies the barrier at $t_i$. For $t \geq 0$, write
\[
    \Rc_t^\beta
    :=
    \{y\in\mathbb R:(t,y)\in\Rc^\beta\}.
\]

\begin{proposition}
\label{prop:tc-direct-tail-recursion}
We have $u^\beta(t_0,\cdot)-v^\xi(t_0,\cdot)=0$.
For \(i=1,\ldots,m\) and \(y\in\mathbb R\), we have
\begin{equation}
\label{eq:tc-upper-tail-cap-direct}
\begin{aligned}
    u^\beta(t_i,y)-v^\xi(t_i,y)
    =
    \min\Bigg\{
        w^\beta(t_i),\,
        &\E^y\!\left[
            u^\beta\!\left(
                t_{i-1},Y_{t_i-t_{i-1}}
            \right)
            -
            v^\xi\!\left(
                t_{i-1},Y_{t_i-t_{i-1}}
            \right)
        \right]\\
        &\quad+
        \P\!\left(
            t_{i-1}<\sigma^\xi\le t_i,\,
            X_{t_i}>y
        \right)
    \Bigg\}.
\end{aligned}
\end{equation}
For every \(i=0,\ldots,m\), the map $
    y
    \longmapsto
    u^\beta(t_i,y)-v^\xi(t_i,y)$
is continuous and non-increasing, with
\[
    \lim_{y\to-\infty}
    \bigl[
        u^\beta(t_i,y)-v^\xi(t_i,y)
    \bigr]
    =
    w^\beta(t_i),
    \qquad
    \lim_{y\to+\infty}
    \bigl[
        u^\beta(t_i,y)-v^\xi(t_i,y)
    \bigr]
    =
    0.
\]
Set \(b_0:=+\infty\). For \(i=1,\ldots,m\), let $
    b_i
    :=
    \sup\left\{
        y\in\mathbb R:
        u^\beta(t_i,y)-v^\xi(t_i,y)=w^\beta(t_i)
    \right\}$. Then
   $\Rc^\beta_{t_i}
    =
    (-\infty,b_i]$.
\end{proposition}

\begin{proof}
Since \(w^\beta(t_0)=0\), Lemma~\ref{lem::uvw} gives
\(u^\beta(t_0,\cdot)=v^\xi(t_0,\cdot)\), and hence
\(u^\beta(t_0,\cdot)-v^\xi(t_0,\cdot)=0\). Fix \(i\ge1\). By Lemma~\ref{lem:tc-grid-reduction}, we have the
recursive relationship:
$$
    u^\beta(t_i,y)
    =
    \min\Big\{
        v^\xi(t_i,y)+w^\beta(t_i), \,\E^y\!\left[
            u^\beta\!\left(
                t_{i-1},Y_{t_i-t_{i-1}}
            \right)
        \right]
    \Big\}.
$$
Hence
$u^\beta(t_i,y)-v^\xi(t_i,y)=
    \min\Bigg\{
        w^\beta(t_i),\,
        \E^y\!\left[
            u^\beta\!\left(
                t_{i-1},Y_{t_i-t_{i-1}}
            \right)
        \right]
        -
        v^\xi(t_i,y)
    \Bigg\}$.
Applying the switching identity of Lemma~\ref{lem::BM-switch} to $\E^y\!\left[
        u^\beta\!\left(
            t_{i-1},Y_{t_i-t_{i-1}}
        \right)
    \right]
    -v^\xi(t_i,y)$, we obtain \eqref{eq:tc-upper-tail-cap-direct}.

We now prove the properties of the map 
$y
    \longmapsto
    u^\beta(t_i,y)-v^\xi(t_i,y)$ by induction. Suppose that they hold for $i-1$.
It is clear that the map
$
    y\longmapsto
    \E^y\!\left[
        u^\beta\!\left(
            t_{i-1},Y_{t_i-t_{i-1}}
        \right)
        -
        v^\xi\!\left(
            t_{i-1},Y_{t_i-t_{i-1}}
        \right)
    \right]$
is continuous and non-increasing. The probability term $\P(t_{i-1} < \sigma^\xi \leq t_i, X_{t_i} > y)$ has the same properties; hence \eqref{eq:tc-upper-tail-cap-direct} gives the asserted continuity and monotonicity.

For the asymptotic behaviours of this map, using the induction hypothesis and dominated convergence, the second
entry of the minimum tends to
\(w^\beta(t_i)+\beta(\{t_i\})\) when \(y\to-\infty\), and to \(0\) when
\(y\to+\infty\).
It follows that
$$
    \lim_{y\to-\infty}
    \bigl[
        u^\beta(t_i,y)-v^\xi(t_i,y)
    \bigr]
    =
    w^\beta(t_i), \ \mbox{and}
     \lim_{y\to+\infty}
    \bigl[
        u^\beta(t_i,y)-v^\xi(t_i,y)
    \bigr]
    =
    0.
$$
This closes the induction.

Finally,
\[
    (t_i,y)\in\Rc^\beta
    \iff
    u^\beta(t_i,y)-v^\xi(t_i,y)=w^\beta(t_i).
\]
Since the difference is continuous, non-increasing, and bounded above
by \(w^\beta(t_i)\), $\Rc_{t_i}^\beta$ is a closed lower interval. For \(i\ge1\), as
\(\beta(\{t_i\})>0\), the second entry of the minimum in
\eqref{eq:tc-upper-tail-cap-direct} has limit
\[
    w^\beta(t_i)+\beta(\{t_i\})
    >
    w^\beta(t_i)
\]
when \(y\to-\infty\), hence $\Rc_{t_i}^\beta$ is non-empty. If
\(w^\beta(t_i)>0\), the same entry tends to
\(0<w^\beta(t_i)\) when \(y\to+\infty\), and hence
\(b_i\in\mathbb R\). If \(w^\beta(t_i)=0\),
Lemma~\ref{lem::uvw} gives
\(u^\beta(t_i,\cdot)-v^\xi(t_i,\cdot)=0\), so
\(b_i=+\infty\).
\end{proof}

\begin{remark}[Comparison with \cite{EJ}]
Lemma~\ref{lem:tc-grid-reduction} parallels the restriction to finitely valued stopping times in the proof of Ekstr\"om and Janson~\cite[Theorem~4.2]{EJ}. In their problem, the
reward on $\{\r < t\}$ is a piecewise constant survival function, so rounding can be compared
pathwise. In the present problem, the payoff
\[
    v^\xi(t-\rho,Y_\rho)+w^\beta(t-\rho)\mathbf 1_{\{\rho<t\}}
\]
also depends on the random state $Y_\rho$ at the stopping time. The required comparison therefore holds in expectation by applying the optional sampling theorem separately on the intervals determined by consecutive points of $\Lambda$, on which $g^\beta$ is constant.

Proposition~\ref{prop:tc-direct-tail-recursion} parallels
\cite[Section~4.1]{EJ}, where they construct the one-marginal barrier
directly by choosing its levels successively to match the given
survival function. Here we instead start from the
stopping region \(\Rc^\beta\). The recursion
\eqref{eq:tc-upper-tail-cap-direct} determines \(b_i\), the endpoint of
\(\Rc_{t_i}^\beta\), and \eqref{eq:tc-grid-active-upper-tail} provides
a probability representation of
\(u^\beta(t_i,\cdot)-v^\xi(t_i,\cdot)\). The additional probability
term in \eqref{eq:tc-upper-tail-cap-direct} accounts for paths
satisfying \(t_{i-1}<\sigma^\xi\le t_i\), which are absent from the
one-marginal construction.
\end{remark}

\subsubsection{Characterization of the barrier outside \texorpdfstring{\(\Lambda\)}{Lambda}}
\begin{lemma}
\label{lem:tc-open-layer-contact}
If \(t\notin \Lambda\) and \(w^\beta(t)>0\), then $\Rc_t^\beta=\varnothing$.
\end{lemma}

\begin{proof}
Suppose that $t_i<t<t_{i+1}$ for some \(i\in\{0,\ldots,m\}\). Combining Lemma~\ref{lem:tc-grid-reduction} with
the strong Markov property, we have
\begin{equation}
\label{eq:tc-open-layer-value}
    u^\beta(t,y)
    =
    \E^y\!\left[
        u^\beta(t_i,Y_{t-t_i})
    \right].
\end{equation}
Subtracting \(v^\xi(t,y)\) from
\eqref{eq:tc-open-layer-value} and using the switching identity, we obtain
\[
\begin{aligned}
    u^\beta(t,y)-v^\xi(t,y)=
    \E^y\!\left[
        u^\beta(t_i,Y_{t-t_i})
        -
        v^\xi(t_i,Y_{t-t_i})
    \right]+
    \P(t_i<\sigma^\xi\le t,X_t>y).
\end{aligned}
\]
Since \(\beta\) has no mass in \((t_i,t_{i+1})\), $w^\beta(t)
    =
    w^\beta(t_i)+\xi((t_i,t])$.
It follows that
\[
\begin{aligned}
    w^\beta(t)
    -
    \bigl[
        u^\beta(t,y)-v^\xi(t,y)
    \bigr] &=
    \left(w^\beta(t_i)
    -
    \E^y\!\left[
        u^\beta(t_i,Y_{t-t_i})
        -
        v^\xi(t_i,Y_{t-t_i})
    \right]\right)\\
    &\quad+
    \P(t_i<\sigma^\xi\le t,X_t\le y).
\end{aligned}
\]
It is clear that both terms on the right-hand side are non-negative. We show below that at
least one term is strictly positive. If \(w^\beta(t_i)>0\), then
\[
    \lim_{z\to+\infty}
    \bigl[
        u^\beta(t_i,z)-v^\xi(t_i,z)
    \bigr]
    =
    0.
\]
Choose \(M>0\) such that $
    u^\beta(t_i,M)-v^\xi(t_i,M)
    <
    w^\beta(t_i)$. By Proposition~\ref{prop:tc-direct-tail-recursion},
the map
    $z\longmapsto u^\beta(t_i,z)-v^\xi(t_i,z)$
is non-increasing and bounded above by \(w^\beta(t_i)\). Therefore
\[
\begin{aligned}
    \E^y\!\left[
        u^\beta(t_i,Y_{t-t_i})
        -
        v^\xi(t_i,Y_{t-t_i})
    \right]
    &\le
    w^\beta(t_i)\P^y(Y_{t-t_i}<M)+
    \bigl[
        u^\beta(t_i,M)-v^\xi(t_i,M)
    \bigr]
    \P^y(Y_{t-t_i}\ge M)\\
    &<
    w^\beta(t_i).
\end{aligned}
\]

If \(\xi((t_i,t])>0\), atomlessness of \(\xi\) gives
\(\P(t_i<\sigma^\xi<t)>0\). By the strong Markov property,
\[
\begin{aligned}
    \P(t_i<\sigma^\xi\le t,X_t\le y)=
    \E\!\left[
        \mathbf 1_{\{t_i<\sigma^\xi<t\}}
        \Phi\!\left(
            \frac{y-X_{\sigma^\xi}}
                 {\sqrt{t-\sigma^\xi}}
        \right)
    \right]
    >0.
\end{aligned}
\]
Finally, $w^\beta(t)
    =
    w^\beta(t_i)+\xi((t_i,t])$ and \(w^\beta(t)>0\) imply
that either \(w^\beta(t_i)>0\) or \(\xi((t_i,t])>0\). In conclusion, we have
$u^\beta(t,y)-v^\xi(t,y)
    <
    w^\beta(t)$
    for every $y\in\mathbb R$,
and consequently $\Rc_t^\beta=\varnothing$.
\end{proof}

\begin{remark}
\label{rmk:tc-finite-contact-barrier}
Combining Lemma~\ref{lem::uvw}, Proposition~\ref{prop:tc-direct-tail-recursion}, and Lemma~\ref{lem:tc-open-layer-contact}, we can summarize the geometry of the barrier \(\Rc^\beta\): for every \(t\ge0\),
\[
    \Rc_t^\beta
    =
    \begin{cases}
        (-\infty,b_i],
            &t=t_i,\quad i=0,\ldots,m,\\[1mm]
        \mathbb R,
            &t\notin\Lambda\text{ and }w^\beta(t)=0,\\[1mm]
        \varnothing,
            &t\notin\Lambda\text{ and }w^\beta(t)>0.
    \end{cases}
\]
Consequently, \(\Rc^\beta\) is a closed lower barrier. Indeed, \(w^\beta\) is continuous and non-decreasing between consecutive points of \(\Lambda\cup\{r_\xi^\beta\}\) and can only jump downward at
these points; hence \(w^\beta\) is lower semicontinuous. \(\Rc^\beta\) is the finite union of the closed cylinder \(\{t:w^\beta(t)=0\}\times\mathbb R\) and the closed sets
\(\{t_i\}\times(-\infty,b_i]\), \(i=0,\ldots,m\), and is a closed lower barrier.
\end{remark}

\subsubsection{Probabilistic representation at the points of \texorpdfstring{\(\Lambda\)}{Lambda}}
\begin{proposition}
\label{prop:tc-decision-time-path-identification}
For every \(i=0,\ldots,m\) and \(y\in\mathbb R\), we have
\begin{equation}
\label{eq:tc-grid-active-upper-tail}
    u^\beta(t_i,y)-v^\xi(t_i,y)
    =
    \P\!\left(
        \sigma^\xi\le t_i<\sigma^{\beta},\,
        X_{t_i}>y
    \right).
\end{equation}
Moreover, let \(t_{m+1}:=r_\xi^\beta\). Then
\begin{equation}
\label{eq:tc-no-open-layer-hit}
    \P\!\left(
        t_i<\sigma^{\beta}<t_{i+1}
    \right)
    =
    0,
    \qquad
    i=0,\ldots,m.
\end{equation}
Consequently,
\begin{equation}
\label{eq:tc-grid-survivor-path-tail}
    u^\beta(t_i,y)
    =
    \P\!\left(
        \sigma^{\beta}>t_i,\,
        X_{t_i}>y
    \right),
    \qquad
    i=0,\ldots,m.
\end{equation}
\end{proposition}

We first record the implication used in the induction step.

\begin{lemma} \label{lem:differ_atom}
Fix \(j\in\{0,\ldots,m\}\). If
\eqref{eq:tc-grid-active-upper-tail} holds at \(t_j\), then $\P\!\left(
        t_j<\sigma^{\beta}<t_{j+1}
    \right)
    =
    0$.
\end{lemma}
\begin{proof}
Let
$    Z_{j,j+1}
    :=
    \{s\in(t_j,t_{j+1}):w^\beta(s)=0\}.
$
Since \(\beta\) has no mass on \((t_j,t_{j+1})\), we have
    $w^\beta(s)
    =
    w^\beta(t_j)+\xi((t_j,s])$ for $s \in (t_j,t_{j+1})$, and hence \(Z_{j,j+1}\) is an initial interval of $(t_j,t_{j+1})$, possibly empty. We will deal with the two parts $Z_{j,j+1}$ and $(t_j,t_{j+1})\setminus Z_{j,j+1}$ separately.

We first consider $Z_{j,j+1}$. If it is empty, the argument is covered by the second part. If it is not empty, then
    $w^\beta(t_j)=0$, and $\xi(Z_{j,j+1})=0$.
Using the monotone convergence theorem and
\eqref{eq:tc-grid-active-upper-tail}, we get
$$
\begin{aligned}
    \P\!\left(
        \sigma^\xi\le t_j<\sigma^{\beta}
    \right)
    &=
    \lim_{y\to-\infty}
    \P\!\left(
        \sigma^\xi\le t_j<\sigma^{\beta},\,
        X_{t_j}>y
    \right)\\
    &=
    \lim_{y\to-\infty}
    \bigl[
        u^\beta(t_j,y)-v^\xi(t_j,y)
    \bigr]
    =
    w^\beta(t_j)
    =
    0.
\end{aligned}
$$
Because \(Z_{j,j+1}\) is an initial interval and
\(\sigma^{\beta}\ge\sigma^\xi\),
\[
\begin{aligned}
    \{\sigma^{\beta}\in Z_{j,j+1}\}
    \subseteq
    \{\sigma^\xi\le t_j<\sigma^{\beta}\}\cup
    \{\sigma^\xi\in Z_{j,j+1}\}.
\end{aligned}
\]
Since both events on the right-hand side have zero probability, we obtain $\P\!\left(
        \sigma^{\beta}\in Z_{j,j+1}
    \right)=0$.

Now consider the part \((t_j,t_{j+1})\setminus Z_{j,j+1}\). As \(w^\beta>0\),
Lemma~\ref{lem:tc-open-layer-contact} shows that \(\sigma^{\beta}\) cannot fall on this part. Consequently,
    $\P\!\left(
        t_j<\sigma^{\beta}<t_{j+1}
    \right)
    =
    0$.
Combining the two cases proves the claim.
\end{proof}

\begin{proof}[Proof of Proposition~\ref{prop:tc-decision-time-path-identification}]
We obtain \eqref{eq:tc-grid-active-upper-tail} by induction, and remark that \eqref{eq:tc-no-open-layer-hit} is a consequence of \eqref{eq:tc-grid-active-upper-tail} by Lemma \ref{lem:differ_atom}.
As \(w^\beta(t_0)=0\), 
Lemma~\ref{lem::uvw} implies that 
    $u^\beta(t_0,\cdot)-v^\xi(t_0,\cdot)=0$.
Moreover, from Remark \ref{rmk:tc-finite-contact-barrier}, we have
\(\Rc_s^\beta=\mathbb R\) for \(0\le s\le t_0\).
Hence
\(\sigma^{\beta}=\sigma^\xi\) on
\(\{\sigma^\xi\le t_0\}\), and therefore 
    $\P\!\left(
        \sigma^\xi\le t_0<\sigma^{\beta},\,
        X_{t_0}>y
    \right)
    =
    0$, which justifies \eqref{eq:tc-grid-active-upper-tail} at \(i=0\).

We now perform the induction step. Fix
\(i\in\{1,\ldots,m\}\) and assume that
\eqref{eq:tc-grid-active-upper-tail} holds at \(t_{i-1}\).
By Lemma~\ref{lem:differ_atom} with \(j=i-1\), we have
$\P\!\left(
        t_{i-1}<\sigma^{\beta}<t_i
    \right)
    =
    0$.
The event
    $\{\sigma^\xi\le t_{i-1}<\sigma^{\beta}\}$
belongs to \(\mathcal F^X_{t_{i-1}}\). Combining this fact and the induction hypothesis
with the switching identity gives
\[
\begin{aligned}
    \E^y\!\left[
        u^\beta\!\left(
            t_{i-1},Y_{t_i-t_{i-1}}
        \right)
        -
        v^\xi\!\left(
            t_{i-1},Y_{t_i-t_{i-1}}
        \right)
    \right]=
    \P\!\left(
        \sigma^\xi\le t_{i-1}<\sigma^{\beta},\,
        X_{t_i}>y
    \right).
\end{aligned}
\]
Consequently,
\begin{equation}
\label{eq:tc-precut-path-identity}
\begin{aligned}
    &\quad\,\,\E^y\!\left[
        u^\beta\!\left(
            t_{i-1},Y_{t_i-t_{i-1}}
        \right)
        -
        v^\xi\!\left(
            t_{i-1},Y_{t_i-t_{i-1}}
        \right)
    \right]
    +
    \P\!\left(
        t_{i-1}<\sigma^\xi\le t_i,\,
        X_{t_i}>y
    \right)\\
    &=
    \P\!\left(
        \sigma^\xi\le t_i\le\sigma^{\beta},\,
        X_{t_i}>y
    \right).
\end{aligned}
\end{equation}
Combining \eqref{eq:tc-precut-path-identity} with
\eqref{eq:tc-upper-tail-cap-direct}, we have
\begin{equation} \label{eq:mini_eq}
\begin{aligned}
    u^\beta(t_i,y)-v^\xi(t_i,y)=
    \min\Bigl\{
        w^\beta(t_i),\,
        \P\!\left(
            \sigma^\xi\le t_i\le\sigma^{\beta},\,
            X_{t_i}>y
        \right)
    \Bigr\}.
\end{aligned}
\end{equation}
We distinguish two cases below.

\textbf{Case 1.} First we assume \(w^\beta(t_i)=0\). From \eqref{eq:mini_eq}, we have $u^\beta(t_i,y)-v^\xi(t_i,y)=0$ for $y\in\mathbb R$,
and hence \(\Rc_{t_i}^\beta=\mathbb R\). On the event $
    \{\sigma^\xi\le t_i\le\sigma^{\beta}\}$, we have $(t_i,X_{t_i})\in\Rc^\beta$ and hence $\sigma^{\beta}\le t_i$. As the same event also implies
$t_i\le\sigma^{\beta}$, it follows that
\(\sigma^{\beta}=t_i\). Therefore
\[
    \P\!\left(
        \sigma^\xi\le t_i<\sigma^{\beta},\,
        X_{t_i}>y
    \right)
    =
    0
    =
    u^\beta(t_i,y)-v^\xi(t_i,y).
\]
Hence \eqref{eq:tc-grid-active-upper-tail} is valid.

\textbf{Case 2.} Now suppose that \(w^\beta(t_i)>0\). By Proposition~\ref{prop:tc-direct-tail-recursion}, we have
    $b_i\in\mathbb R$ and $\Rc_{t_i}^\beta=(-\infty,b_i]$. From \eqref{eq:mini_eq} we know that $\P\!\left(
        \sigma^\xi\le t_i\le\sigma^{\beta},\,
        X_{t_i}>y
    \right)$ is at least \(w^\beta(t_i)\) for
\(y\le b_i\), and strictly less than \(w^\beta(t_i)\) for
\(y>b_i\). By the continuity of the map
    $y
    \longmapsto
    \P\!\left(
        \sigma^\xi\le t_i\le\sigma^{\beta},\,
        X_{t_i}>y
    \right)$, it follows that
$$
    \P\!\left(
        \sigma^\xi\le t_i\le\sigma^{\beta},\,
        X_{t_i}>b_i
    \right)
    =
    w^\beta(t_i).
$$
As
    $\{\sigma^\xi\le t_i<\sigma^{\beta}\}
    =
    \{
        \sigma^\xi\le t_i\le\sigma^{\beta},\,
        X_{t_i}>b_i
    \}$ almost surely,
we have, for every \(y\in\mathbb R\),
$$
\begin{aligned}
    \P\!\left(
        \sigma^\xi\le t_i<\sigma^{\beta},\,
        X_{t_i}>y
    \right)
    =
    \begin{cases}
        w^\beta(t_i),
            &y\le b_i,\\[1mm]
        \P\!\left(
            \sigma^\xi\le t_i\le\sigma^{\beta},\,
            X_{t_i}>y
        \right),
            &y>b_i.
    \end{cases}
\end{aligned}
$$
By \eqref{eq:mini_eq}, the right-hand side equals
\(u^\beta(t_i,y)-v^\xi(t_i,y)\) in both cases. This proves
\eqref{eq:tc-grid-active-upper-tail} at \(t_i\) and closes the
induction.


Finally, since $\sigma^{\beta}\ge\sigma^\xi$, the event $\{\sigma^{\beta}>t_i\}
    =
    \{\sigma^\xi>t_i\}
    \mathbin{\cup}
    \{\sigma^\xi\le t_i<\sigma^{\beta}\}$ is the disjoint union of the two events on the right-hand side.
By \eqref{eq:tc-grid-active-upper-tail}, we have
$$
    \P\!\left(
        \sigma^{\beta}>t_i,\,
        X_{t_i}>y
    \right)
    =
    v^\xi(t_i,y)
    +
    u^\beta(t_i,y)-v^\xi(t_i,y)
   =
    u^\beta(t_i,y).
$$
This proves \eqref{eq:tc-grid-survivor-path-tail}.
\end{proof}

We end this section with the proof of Proposition~\ref{prop:tc-grid-value-contact}.
\begin{proof}[Proof of Proposition~\ref{prop:tc-grid-value-contact}]
We split the proof into three cases.

\textbf{Case 1.} $t=t_i$ for some $i=0,1,\ldots,m$. In this case, \eqref{eq:tc-grid-value-direct} follows directly from Proposition~\ref{prop:tc-decision-time-path-identification}.

\textbf{Case 2.} $t\notin\Lambda$ and $w^\beta(t)=0$. From Remark~\ref{rmk:tc-finite-contact-barrier}, $\Rc_t^\beta=\mathbb R$.
By definition, $\sigma^{\beta}\ge\sigma^\xi$. On the other hand, $w^\beta(t)=0$ together with $\sigma^\xi\le t$ implies $\sigma^{\beta}\le t$.
Consequently,
$$
    \{\sigma^{\beta}>t\}
    =
    \{\sigma^\xi>t\}.
$$
By Lemma~\ref{lem::uvw}, $u^\beta(t,\cdot)=v^\xi(t,\cdot)$, and hence
$$
    u^\beta(t,x)
    =
    \P(\sigma^{\beta}>t,X_t>x).
$$

\textbf{Case 3.} $t\notin\Lambda$ and $w^\beta(t)>0$. In particular, there exists $i \in \{0,1,\ldots,m\}$ such that
   $t_i<t<t_{i+1}$, and $w^\beta(t)>0$. By \eqref{eq:tc-no-open-layer-hit}, we have $\{\sigma^{\beta}>t\}
    =
    \{\sigma^{\beta}>t_i\}$ almost surely. Now combining \eqref{eq:tc-open-layer-value}, \eqref{eq:tc-grid-survivor-path-tail}, and
the switching identity, we obtain
$$
    u^\beta(t,x)
    =
    \E^x\!\left[
        u^\beta(t_i,Y_{t-t_i})
    \right]=
    \P(
        \sigma^{\beta}>t_i,\,
        X_t>x
    )=
    \P(
        \sigma^{\beta}>t,\,
        X_t>x
    ).
$$
Combining the three cases gives $u^\beta=v^\beta$. For $i=0,\ldots,m$, monotone convergence in
\eqref{eq:tc-grid-active-upper-tail} and Proposition~\ref{prop:tc-direct-tail-recursion} give
$$
\P(\sigma^\xi\le t_i<\sigma^\beta)
=
\lim_{y\to-\infty}
\bigl[u^\beta(t_i,y)-v^\xi(t_i,y)\bigr]
=
w^\beta(t_i).
$$
Since $\sigma^\beta\ge\sigma^\xi$, we have $\P(\sigma^\beta>t_i) = g^\xi(t_i)+w^\beta(t_i) = g^\beta(t_i)$.
If $t\notin\Lambda$ and $w^\beta(t)=0$, Case~2 gives $\P(\sigma^\beta>t)=\P(\sigma^\xi>t)=g^\xi(t)=g^\beta(t)$.
If $t_i<t<t_{i+1}$ and $w^\beta(t)>0$, then \eqref{eq:tc-no-open-layer-hit} and $\beta((t_i,t])=0$ give
$\P(\sigma^\beta>t)=\P(\sigma^\beta>t_i)=g^\beta(t_i)=g^\beta(t)$.
Thus \(\P(\sigma^\beta>t)=g^\beta(t)\) for all \(t\ge0\), and $\sigma^\beta\sim\beta$.

\end{proof}

\subsection{Proof of Theorem~\ref{thm::main} in the general case}
\label{sec:general_case}

In this section, we prove Theorem~\ref{thm::main} in the general case using the approximation arguments and the results of Section~\ref{sec:finite_support}. Technical results related to compactness and stability are collected in
Appendix~\ref{app:lower-barrier-compactness}.

\subsubsection{Dyadic approximation}

For \(m\ge1\), set $\pi_m^-(t)
    :=
    2^{-m}\lfloor2^m t\rfloor$ and
    $g_m(t)
    :=
    g^\beta\bigl(\pi_m^-(t)\bigr)$.
The function $g_m$ is clearly a survival function. Let
\(\beta^m\) be the corresponding probability measure on \(\mathbb R_+\), and introduce $w_m:=g_m-g^\xi$.

\begin{lemma}
\label{lem:canonical-upper-approximation}
For every \(m\geq 1\), the measure \(\beta^m\) is supported on the set of dyadic points $\{k\,2^{-m}:k\in\mathbb N_0\}$,
and this set meets every bounded interval in finitely many points.
Moreover, we have
$$
    \xi
    \le_{\mathrm{st}}
    \beta
    \le_{\mathrm{st}}
    \beta^m.
$$
For every \(T<\infty\), let
$\varepsilon_m(T)
    :=
    \sup_{0\le t\le T}
    \bigl(g_m(t)-g^\beta(t)\bigr)$.
It follows that $\varepsilon_m(T)\to0$ as $m\to\infty$.
In particular, $\beta^m$ converges to $\beta$ weakly.
\end{lemma}

\begin{proof}
The function \(g_m\) is constant on all intervals $[k2^{-m},(k+1)2^{-m})$, $k=0,1,2,\ldots$.
Hence \(\beta^m\) is supported on the set of dyadic points
$\{k\,2^{-m}:k=0,1,2,\ldots\}$.

As \(\pi_m^-(t)\le t\) and \(g^\beta\) is non-increasing, $g_m(t)
    =
    g^\beta\bigl(\pi_m^-(t)\bigr)
    \ge
    g^\beta(t)
    \ge
    g^\xi(t)$,
and hence the stochastic ordering holds.

Finally, noting that
    $0
    \le
    g_m(t)-g^\beta(t)
    =
    g^\beta\bigl(\pi_m^-(t)\bigr)-g^\beta(t)
$
and
$
    0\le t-\pi_m^-(t)<2^{-m},
$
by uniform continuity of \(g^\beta\) on compact intervals, we obtain
\(\varepsilon_m(T)\to0\) for every \(T\ge0\). Thus $\beta^m$ converges to $\beta$ weakly.
\end{proof}

For the target \(\beta^m\), define
  $$  u_m:=u^{\beta^m},
    \qquad
    \Rc_m:=\Rc^{\beta^m},
\qquad
    \sigma_m
    :=
    \inf\{t\ge\sigma^\xi:(t,X_t)\in\Rc_m\}.
$$

\begin{lemma}
\label{lem:atomic-barrier-identities}
For every \(m\geq 1\), the set \(\Rc_m\) is a closed lower
barrier, and \(\sigma_m\) is an
\((\mathcal F_t^X)\)-stopping time. Moreover, we have $\sigma_m\sim\beta^m$, and
\[
    u_m(t,x)
    =
    \P(\sigma_m>t,\ X_t>x),
    \qquad t\ge0,\quad x\in\mathbb R.
\]
Consequently,
   $ u_m(t,x)-v^\xi(t,x)
    =
    \P\left(
        \sigma^\xi\le t<\sigma_m,\ X_t>x
    \right)$.
\end{lemma}

\begin{proof}
Fix \(m\) and \(T>0\), and define
\[
    g_m^T(s)
    :=
    \begin{cases}
        g_m(s),&0\le s<T,\\
        g^\xi(s),&s\ge T.
    \end{cases}
\]
It is clear that \(g_m^T\) is non-increasing and right-continuous. At \(T\), its downward jump has size
$g_m(T-)-g^\xi(T)\ge0$. Moreover, $g_m^T(0)=1$ and $\lim_{s\to\infty}g_m^T(s)=0$.
Hence \(g_m^T\) is the survival function of some probability measure,
which we denote by \(\beta^{m,T}\). As $g_m^T(t)\ge g^\xi(t)$ for $t\ge0$, we have
$\xi\le_{\mathrm{st}}\beta^{m,T}$.
In addition, \(\beta^{m,T}\) has a finite first moment, since
$$
    \int_0^\infty g_m^T(s)\,\d s
    \le
    T+\int_T^\infty g^\xi(s)\,\d s
    <
    \infty.
$$
We introduce
$$
    w_m^T:=g_m^T-g^\xi,
    \quad
    u_m^T:=u^{\beta^{m,T}},
    \quad
    \Rc_m^T:=\Rc^{\beta^{m,T}},
    \quad
    \sigma_m^T := \inf\{t\ge\sigma^\xi:(t,X_t)\in\Rc_m^T\}.
$$
By Remark~\ref{rmk:tc-finite-contact-barrier} and
Proposition~\ref{prop:tc-grid-value-contact}, \(\Rc_m^T\) is a closed lower barrier,
$
    \sigma_m^T\sim\beta^{m,T}
$,
and
$
    u_m^T(t,x)
    =
    \P(\sigma_m^T>t,\ X_t>x)
   $
for all $t \geq 0$ and $x \in \R$.


Next, we relate the localized and original optimal stopping problems.
For every \(0\le t<T\), as
    $w_m^T\big|_{[0,t]}
    =
    w_m\big|_{[0,t]}$,
we have
$
    u_m^T(t,x)=u_m(t,x)
$.
It follows that
\[
    \Rc_m^T\cap([0,T)\times\mathbb R)
    =
    \Rc_m\cap([0,T)\times\mathbb R).
\]
Since the barriers agree before $T$, both hitting times agree up to time $T$, i.e.,
   $ \sigma_m^T\wedge T
    =
    \sigma_m\wedge T$.

We now pass from the localized barriers to $\Rc_m$. Fix \(t\ge0\) and \(T>t\). To prove the closedness, let
\((t_k,x_k)\in\Rc_m\) converge to \((t,x)\).
Then \(t_k<T\) for all sufficiently large \(k\), so
\((t_k,x_k)\in\Rc_m^T\). Since \(\Rc_m^T\) is closed,
\((t,x)\in\Rc_m^T\), and \(t<T\) gives
\((t,x)\in\Rc_m\). Thus \(\Rc_m\) is a closed lower
barrier. It is clear that \(\sigma_m\) is a stopping time.

It remains to identify its law and value function. Fix \(t\ge0\) and
\(x\in\mathbb R\), and choose \(T>t\). The localized identities and $\sigma_m^T \wedge T = \sigma_m \wedge T$ give
\[
    u_m(t,x)
    =
    \P(\sigma_m>t,\ X_t>x),
    \qquad
    \P(\sigma_m>t)=g_m(t).
\]
Finally, since \(\sigma_m\ge\sigma^\xi\), we have
\[
    u_m(t,x)-v^\xi(t,x)
    =
    \P(\sigma_m>t,\ X_t>x)
    -
    \P(\sigma^\xi>t,\ X_t>x)
   =
    \P\left(
        \sigma^\xi\le t<\sigma_m,\ X_t>x
    \right).
\]
\end{proof}

\begin{lemma}
\label{lem:atomic-value-stability}
For every \(m\ge1\) and \(T<\infty\), we have
\[
    0
    \le
    u_m(t,x)-u^\beta(t,x)
    \le
    \varepsilon_m(T),
    \qquad
    0\le t\le T,\quad x\in\mathbb R.
\]
In particular,
\[
    \sup_{\substack{0\le t\le T\\x\in\mathbb R}}
    \left|
        \bigl[u_m(t,x)-v^\xi(t,x)\bigr]
        -
        \bigl[u^\beta(t,x)-v^\xi(t,x)\bigr]
    \right|
    \le
    \varepsilon_m(T)
    \overset{m \to \infty}{\longrightarrow}0.
\]
\end{lemma}

\begin{proof}
Fix $T<\infty$, $(t,x) \in [0, T] \times \R$, and $\tau \in \TT^t$. By Lemma~\ref{lem:canonical-upper-approximation}, the payoffs defining $u_m(t,x)$ and $u^\beta(t,x)$, evaluated at the same $\tau$, differ by
\[
    \E^x\left[
        \bigl(
            g_m(t-\tau)-g^\beta(t-\tau)
        \bigr)
        \mathbf 1_{\{\tau<t\}}
    \right]
    \in
    [0,\varepsilon_m(T)].
\]
Taking the infimum over \(\tau\in\TT^t\) proves the first assertion. The second follows because $v^\xi$ is independent of $m$.
\end{proof}

\subsubsection{Limiting arguments}

In this section, we use the preceding approximation results and pass to the limit.
For each \(m\), let
\(
    b_m:\mathbb R_+\longrightarrow[-\infty,\infty]
\)
be the upper semicontinuous boundary of \(\Rc_m\), so that
\(
    \Rc_m
    =
    \{(t,x):x\le b_m(t)\}.
\)
\begin{proposition}
\label{prop:general-limiting-barrier}
There exists a subsequence, still denoted by $(b_m)$, and an upper semicontinuous \(b:\mathbb R_+\to[-\infty,\infty]\) such that
\(b_{m}\xrightarrow{\Gamma^+}b\). Let
\[
    \widetilde\Rc:=\{(t,x):x\le b(t)\},
    \qquad
    \widehat\sigma:=
    \inf\{t\ge\sigma^\xi:(t,X_t)\in\widetilde\Rc\}.
\]
Then \(\widetilde\Rc\) is a closed lower barrier, and
    $\sigma_{m}\longrightarrow\widehat\sigma$ almost surely.
Moreover, we have
\begin{equation}
\label{eq:limiting-law-value}
    \widehat\sigma\sim\beta,
    \qquad
    u^\beta(t,x)
    =
    \P\left(\widehat\sigma>t,\ X_t>x\right),
    \quad
    \text{for all }t\ge0, \ x\in\mathbb R,
\end{equation}
and
\begin{equation}
\label{eq:limiting-active-upper-tail}
    u^\beta(t,x)-v^\xi(t,x)
    =
    \P\left(
        \sigma^\xi\le t<\widehat\sigma,\ X_t>x
    \right),
    \quad
    \text{for all }t\ge0,\ x\in\mathbb R.
\end{equation}
\end{proposition}

\begin{proof}
Proposition~\ref{prop:hypograph-Hausdorff-compactness} gives a subsequence of $(b_m)$ and an upper semicontinuous function $b$ with $b_{m} \xrightarrow{\Gamma^+} b$. Fix this subsequence.
By Theorem~\ref{thm:delayed-lower-hit-stability}, we have $\sigma_{m}\longrightarrow\widehat\sigma$ almost surely in $[0,\infty]$. The upper semicontinuity of $b$ makes $\widetilde\Rc$ a closed lower barrier.
By construction, \(\widehat\sigma\ge\sigma^\xi\).

View the laws of $\sigma_m$ and $\widehat\sigma$ as probability measures on $[0,\infty]$. The almost sure convergence $\sigma_m\to\widehat\sigma$ implies
\(\mathcal L(\sigma_m)\Rightarrow\mathcal L(\widehat\sigma)\).
Lemma~\ref{lem:atomic-barrier-identities} gives
\(\mathcal L(\sigma_m)=\beta^m\), while Lemma~\ref{lem:canonical-upper-approximation} gives
\(\beta^m\Rightarrow\beta\). Uniqueness of the weak limit therefore yields
    $\widehat\sigma\sim\beta$.
In particular, \(\widehat\sigma<\infty\) almost surely.

Now let us pass to the limit. Fix \(t\ge0\) and
\(x\in\mathbb R\). Since \(\widehat\sigma\sim\beta\) and \(\beta\) is
atomless,
\(\P(\widehat\sigma=t)=0\). Hence
\[
    \mathbf 1_{\{
        \sigma_{m}>t,\,X_t>x
    \}}
    \longrightarrow
    \mathbf 1_{\{
        \widehat\sigma>t,\,X_t>x
    \}}
    \qquad\text{a.s.}.
\]
The bounded convergence theorem, together with Lemmas~\ref{lem:atomic-barrier-identities}
and~\ref{lem:atomic-value-stability}, applied on any
\([0,T]\) with \(T\ge t\), gives
\[
\begin{aligned}
    u^\beta(t,x)
    =
    \lim_{m\to\infty}u_{m}(t,x)
    =
    \P(\widehat\sigma>t,X_t>x).
\end{aligned}
\]
Since
\(\widehat\sigma\ge\sigma^\xi\), we have the disjoint union
    $\{\widehat\sigma>t\}
    =
    \{\sigma^\xi>t\}
    \mathbin{\cup}
    \{\sigma^\xi\le t<\widehat\sigma\}$.
Together with \eqref{eq:limiting-law-value} and $v^\xi(t,x) = \P(\sigma^\xi > t, X_t > x)$, we obtain \eqref{eq:limiting-active-upper-tail}.
\end{proof}

\begin{corollary}
\label{cor:general-contact-barrier}
The map
\[
    (t,x)
    \longmapsto
    u^\beta(t,x)-v^\xi(t,x)
\]
is jointly continuous. Consequently, \(\Rc^\beta\) is a
closed lower barrier.
\end{corollary}

\begin{proof}
Let \((t_n,x_n)\to(t,x)\) with \(t>0\). The continuity of $X$ and the atomlessness of the laws of $\sigma^\xi$, $\widehat\sigma$, and $X_t$ give
\[
    \mathbf 1_{\{
        \sigma^\xi\le t_n<\widehat\sigma,\,
        X_{t_n}>x_n
    \}}
    \longrightarrow
    \mathbf 1_{\{
        \sigma^\xi\le t<\widehat\sigma,\,
        X_t>x
    \}}
    \qquad\text{a.s.}
\]
Thus \eqref{eq:limiting-active-upper-tail} and the bounded convergence theorem give continuity at every $(t,x)$ with $t > 0$. Moreover, Lemma~\ref{lem::uvw} gives
\[
    0
    \le
    u^\beta(t,x)-v^\xi(t,x)
    \le
    w^\beta(t)
    \longrightarrow
    0
\]
when \(t\downarrow0\), uniformly in \(x\), while $u^\beta(0,x)-v^\xi(0,x)=0$. This proves joint continuity on \(\mathbb R_+\times\mathbb R\).

For each \(t \geq 0\), \eqref{eq:limiting-active-upper-tail} shows that
\(x\mapsto u^\beta(t,x)-v^\xi(t,x)\) is non-increasing, and
Lemma~\ref{lem::uvw} implies that
$
    u^\beta(t,x)-v^\xi(t,x)
    \le
    w^\beta(t)$.
If $(t,x) \in \Rc^\beta$ and \(y\le x\), then
\[
    w^\beta(t)
    \ge
    u^\beta(t,y)-v^\xi(t,y)
    \ge
    u^\beta(t,x)-v^\xi(t,x)
    =
    w^\beta(t).
\]
Thus $(t,y) \in \Rc^\beta$, so $\Rc^\beta$ is a lower barrier. Finally, $\Rc^\beta = \{(t,x):u^\beta(t,x)-v^\xi(t,x)=w^\beta(t)\}$ is closed because both sides of the equality are continuous.
\end{proof}

\begin{remark}
In \cite[Theorem~5.1]{EJ}, the authors deal with similar approximation arguments in the one-marginal context. In particular, they apply a one-sided version of
Anulova's argument in \cite{Anulova1980}. Anulova first ignores the barrier before a fixed time \(s>0\) and then
moves its boundary by \(\varepsilon\). The associated hitting times
converge to the original hitting time as \(\varepsilon\downarrow0\).
Together with the subsequent limit \(s\downarrow0\), this yields a
limiting barrier whose hitting time has the target survival function.

Here the next hit may occur at the random entrance time
\(\sigma^\xi\), so we cannot ignore any fixed interval after the
entrance. We instead replace \(X_t\) by \(X_t-\varepsilon t\).
Lemma~\ref{lem:shifted-iterated-gate} shows that the entrance time
generated by the modified path does not exceed \(\sigma^\xi\) almost
surely. Lemma~\ref{lem:weak-strict-lower-hit} shows that the weak and
strict lower-barrier hitting times coincide almost surely.
If the hit occurs at \(\sigma^\xi\), the approximating times must satisfy
\[
    r_k>\sigma^\xi,\qquad
    r_k\downarrow\sigma^\xi,\qquad
    X_{r_k}<b(r_k).
\]
The only possible obstruction is an immediate strict drop of \(b\)
after the entrance. Lemma~\ref{lem:right-discontinuity-countable}
shows that such drops can occur only at countably many deterministic
times. Since \(\sigma^\xi\sim\xi\) and \(\xi\) is atomless, the
entrance avoids these times almost surely. Thus the compactness
argument extends to a random entrance time.
\end{remark}

\subsubsection{Identification of the hitting time}

In this section, we prove that
\[
    \widehat\sigma
    =
    \inf\{t\ge\sigma^\xi:(t,X_t)\in\Rc^\beta\}
    =
    \sigma^{\beta}
    \qquad\text{a.s.}.
\]
We first identify the sections of \(\Rc^\beta\) with those determined by a function
\(b\) when \(w^\beta>0\), and then show that filling the sections
with \(w^\beta=0\) does not change the hitting time.

Let
\[
    \mathcal O^\beta:=\{t\ge0:w^\beta(t)>0\},
    \qquad
    \mathcal Z^\beta:=\{t\ge0:w^\beta(t)=0\}.
\]
These sets are respectively open and closed because \(w^\beta\) is
continuous.

\begin{proposition}
\label{prop:limiting-active-support}
For every \(t\in\mathcal O^\beta\) and \(x\in\mathbb R\),
\[
    u^\beta(t,x)-v^\xi(t,x)=w^\beta(t)
    \quad\Longleftrightarrow\quad
    x\le b(t).
\]
\end{proposition}

\begin{proof}
Using the laws of \(\widehat\sigma\) and \(\sigma^\xi\) and the fact that
\(\widehat\sigma\ge\sigma^\xi\), we have
\[
    \P(\sigma^\xi\le t<\widehat\sigma)
    =
    g^\beta(t)-g^\xi(t)
    =
    w^\beta(t).
\]
We first suppose that \(x\le b(t)\). On the event
\(\{\sigma^\xi\le t<\widehat\sigma\}\), the definition of
\(\widehat\sigma\) implies that \(X_t>b(t)\). Hence
\[
    u^\beta(t,x)-v^\xi(t,x)
    =
    \P(\sigma^\xi\le t<\widehat\sigma,\ X_t>x)
    =
    w^\beta(t),
\]
where the first equality is by
\eqref{eq:limiting-active-upper-tail}.

Conversely, let \(x>b(t)\), and choose
\(z\in(b(t),x)\). Since $t \in \mathcal{O}^\beta$ and \(w^\beta(0)=0\), one has \(t>0\). The
openness of \(\mathcal O^\beta\) allows us to choose \(0<s<t\) such that
    $[s,t]\subset\mathcal O^\beta$.
For every \(r\in[s,t]\),
\[
    \P(\sigma^\xi\le r<\widehat\sigma)
    =
    w^\beta(r)>0.
\]
It follows that \(b(r)<+\infty\) on \([s,t]\); otherwise the event on
the left would be empty. Since \(b\) is upper semicontinuous, it is
therefore bounded from above on \([s,t]\). The subprobability measure $A \mapsto \P(\sigma^\xi\le s<\widehat\sigma, X_s\in A)$ is nonzero and concentrated on $(b(s),\infty)$. Its support therefore contains some \(y>b(s)\).

We claim that there is a continuous function
\(q:[s,t]\to\mathbb R\) satisfying
\[
    q(s)=y,
    \qquad
    q(t)=z,
    \qquad
    q(r)>b(r)
    \quad\text{for all }r\in[s,t].
\]
Indeed, if \(h\) is the affine function joining \(y\) to \(z\), then $q(r) := h(r) + C(r-s)(t-r)$ has the required properties for all sufficiently large \(C\).
Upper semicontinuity implies that \(h>b\) near the endpoints, while $b$ is bounded above on the remaining compact subinterval. Since \(q-b\) is lower semicontinuous and strictly positive on
\([s,t]\), choose \(a>0\) such that
\[
    q(r)-b(r)>4a
    \quad\text{for all }r\in[s,t],
    \qquad
    z+2a<x.
\]
By the choice of \(y\), the event
    $E
    :=
    \{\sigma^\xi\le s<\widehat\sigma,\ |X_s-y|<a\}$
has positive probability. The full-support property of Brownian increments gives
\[
    \P\left(
        \sup_{0\le u\le t-s}
        \left|
            (X_{s+u}-X_s)-(q(s+u)-y)
        \right|
        <a
    \right)>0.
\]
This event is
independent of \(\Fc_s^X\), so its intersection with \(E\) has positive probability.
On that intersection,
$|X_r-q(r)|<2a$
  for $s\le r\le t$.
Consequently, \(X_r>b(r)\) on \([s,t]\), while
    $X_t<q(t)+2a=z+2a<x$.
Hence
$
    \P(\sigma^\xi\le t<\widehat\sigma,\ X_t\le x)>0$,
and \eqref{eq:limiting-active-upper-tail} yields
\[
    u^\beta(t,x)-v^\xi(t,x)
    =
    w^\beta(t)
    -
    \P(\sigma^\xi\le t<\widehat\sigma,\ X_t\le x)
    <
    w^\beta(t).
\]
\end{proof}

For \(t\in\mathcal Z^\beta\), Lemma~\ref{lem::uvw} gives
\[
    0
    \le
    u^\beta(t,x)-v^\xi(t,x)
    \le
    w^\beta(t)
    =
    0.
\]
Thus \(\Rc_t^\beta=\mathbb R\).
It remains to show that, almost surely, no such time lies between
\(\sigma^\xi\) and \(\widehat\sigma\).

\begin{proposition}
\label{prop:zero-gap-avoidance}
We have
\[\P\left(
        \mathcal Z^\beta\cap
        [\sigma^\xi,\widehat\sigma)
        \ne\varnothing
    \right)
    =
    0.
\]
\end{proposition}

\begin{proof}
Choose a countable set
\(\mathcal D\subset\mathcal Z^\beta\) that is dense in
\(\mathcal Z^\beta\) with respect to the relative topology. For every
\(d\in\mathcal D\), by \(\sigma^\xi\le\widehat\sigma\) and the laws of the stopping times, we have
\[
    \P(\sigma^\xi\le d<\widehat\sigma)
    =
    \P(\widehat\sigma>d)-\P(\sigma^\xi>d)
    =
    g^\beta(d)-g^\xi(d)
    =
    w^\beta(d)
    =
    0.
\]
If
    $\mathcal Z^\beta\cap(\sigma^\xi,\widehat\sigma)
    \ne\varnothing$,
then this intersection is a nonempty relatively open subset of
\(\mathcal Z^\beta\), and hence contains some \(d\in\mathcal D\).
Therefore
\[
    \left\{
        \mathcal Z^\beta\cap(\sigma^\xi,\widehat\sigma)
        \ne\varnothing
    \right\}
    \subseteq
    \bigcup_{d\in\mathcal D}
    \{\sigma^\xi\le d<\widehat\sigma\},
\]
and the event on the left has probability zero.

It remains to consider the possibility that
\(
    \sigma^\xi\in\mathcal Z^\beta,
\) and
\(
    \sigma^\xi<\widehat\sigma.
\)
Outside the null event considered above, we then have
\[
    (\sigma^\xi,\widehat\sigma)
    \subset\mathcal O^\beta.
\]
Thus \(\sigma^\xi\) is the left endpoint of a connected component of
the open set \(\mathcal O^\beta\). There are at most countably many
such endpoints, whereas $\sigma^\xi$ has the atomless law $\xi$.
This endpoint event therefore has probability zero. Combining the two cases finishes the proof.
\end{proof}

We can now prove Theorem~\ref{thm::main}.

\begin{proof}[Proof of Theorem~\ref{thm::main}]
Proposition~\ref{prop:general-limiting-barrier} and Corollary~\ref{cor:general-contact-barrier} show that
\[
    \widehat\sigma\sim\beta,
    \qquad
    u^\beta(t,x)
    =
    \P(\widehat\sigma>t,\ X_t>x),
\]
and \(\Rc^\beta\) is a closed lower barrier.

From Proposition~\ref{prop:limiting-active-support}, we get
\[
    \Rc_t^\beta
    =
    (-\infty,b(t)],
    \qquad
    t\in\mathcal O^\beta,
\]
whereas \(\Rc_t^\beta=\mathbb R\) for
\(t\in\mathcal Z^\beta\). Hence
\(\widetilde\Rc\subseteq\Rc^\beta\), and therefore $\sigma^{\beta}\le\widehat\sigma$ a.s.

On the event
\(\{\sigma^{\beta}<\widehat\sigma\}\), one has
\(\sigma^{\beta}<\widehat\sigma<\infty\). The closedness of
\(\Rc^\beta\) and the continuity of \(X\) imply that
    $(\sigma^{\beta},X_{\sigma^{\beta}})
    \in
    \Rc^\beta$.
In addition, \(\sigma^{\beta} \notin \mathcal O^\beta\), as the two barriers have the same section
and \(\widehat\sigma\) is the first hit of
\(\widetilde\Rc\). Therefore
\[
    \{\sigma^{\beta}<\widehat\sigma\}
    \subseteq
    \left\{
        \mathcal Z^\beta
        \cap
        [\sigma^\xi,\widehat\sigma)
        \ne\varnothing
    \right\}.
\]
By Proposition~\ref{prop:zero-gap-avoidance}, the event on the
right-hand side has probability $0$. Consequently,
\[
    \sigma^{\beta}=\widehat\sigma
    \qquad\text{a.s.}
\]
The asserted law and value identity now follow from
Proposition~\ref{prop:general-limiting-barrier}.
\end{proof}

\begin{remark}
\label{rem:relative-uniqueness}
We conclude that for fixed $\xi$ and $\sigma^\xi$, the target measure \(\beta\) determines the boundary on \(\mathcal O^\beta\). Fixing its value to \(+\infty\) on
\(\mathcal Z^\beta\) then determines it on all of \(\mathbb R_+\).

Let \(c:\mathbb R_+\to[-\infty,\infty]\) be upper semicontinuous, and set
\[
\Rc(c):=\{(t,x):x\le c(t)\},\qquad
\tau_c:=\inf\{t\ge\sigma^\xi:(t,X_t)\in\Rc(c)\}.
\]
Suppose that \(\tau_c\sim\beta\), and write
\(v^c(t,x):=\P(\tau_c>t,X_t>x)\). Since
\(\tau_c\ge\sigma^\xi\), \(\tau_c\sim\beta\), and
\(\sigma^\xi\sim\xi\), then for every $t \geq 0$ and $x \in \R$,
\[
v^c(t,x)-v^\xi(t,x)
=\P(\sigma^\xi\le t<\tau_c,X_t>x)\le w^\beta(t).
\]
Applying Lemma~\ref{lem::V-submartingale} to \(\tau_c\) and using this
estimate in the optional-sampling argument gives \(v^c\le u^\beta\). Moreover, for
\(x\le c(t)\),
\[
v^c(t,x)-v^\xi(t,x)
=\P(\sigma^\xi\le t<\tau_c)=w^\beta(t).
\]
Hence Lemma~\ref{lem::uvw} gives
\(\Rc(c)\subseteq\Rc^\beta\) and \(\sigma^\beta\le\tau_c\).
Since \(\sigma^\beta\) and \(\tau_c\) both have law \(\beta\),
\(\tau_c=\sigma^\beta=\widehat\sigma\) a.s.; hence \(v^c=u^\beta\).
The converse argument in the proof of
Proposition~\ref{prop:limiting-active-support}, applied to
\((c,\tau_c)\), then yields \(c=b\) on \(\mathcal O^\beta\).

Moreover, if \(c=b\) on \(\mathcal O^\beta\) and
\(c\ge b\) on \(\mathcal Z^\beta\), then
\(\Rc(b)\subseteq\Rc(c)\), so \(\tau_c\le\widehat\sigma\).
Continuity of \(X\), closedness of \(\Rc(c)\), and equality on
\(\mathcal O^\beta\) imply
\[
\{\tau_c<\widehat\sigma\}
\subseteq
\{\mathcal Z^\beta\cap[\sigma^\xi,\widehat\sigma)\ne\varnothing\}.
\]
Proposition~\ref{prop:zero-gap-avoidance} therefore gives
\(\tau_c=\widehat\sigma\) a.s. Thus, among upper semicontinuous boundaries
with \(\tau_c\sim\beta\), the convention \(c=+\infty\) on
\(\mathcal Z^\beta\) gives the unique boundary, with
\(\Rc(c)=\Rc^\beta\).
\end{remark}

\section{Proofs of the optimality results}\label{sec:opti}

To prove the optimality results, we first compare the stopping times $(\sigma_1, \ldots, \sigma_n)$ obtained in Theorem~\ref{thm::goal} with an arbitrary $\r = (\r_1, \ldots, \r_n) \in \TT(\boldsymbol{\mu}_n)$. Theorem~\ref{thm:opt-multi} then follows by combining this comparison with It\^o's formula. For any stopping time $\theta$, set
$$v_\theta(t,x) := \P(\theta > t, X_t > x).$$
The following inequality relates $v_\theta(t,x)$ to the term $\E^x[v_\theta(t-\tau, Y_\tau)]$ appearing in the recursive definition of $u^i$.

\begin{lemma}
\label{lem:opt-ineq}
    Let $\theta \in \TT$ be an almost surely finite stopping time. Then for every $t \geq 0$, $x \in \R$, and $\tau \in \TT^t$,
    $$v_\theta(t,x) \leq \E^x\left[v_\theta(t-\tau, Y_\tau)\right].$$
\end{lemma}

\begin{proof}
    By Lemma~\ref{lem::V-submartingale}, applied with $\theta$ in place of $\sigma^\xi$, the process $\bigl(v_\theta(t-s,Y_s)\bigr)_{0\le s\le t}$ is a bounded submartingale. Since $\tau\le t$ almost surely, one can conclude
    $$
    v_\theta(t,x)=v_\theta(t,Y_0)\le\E^x\!\left[v_\theta(t-\tau,Y_\tau)\right].
    $$
\end{proof}

\begin{proposition}
    For every $\r \in \TT(\boldsymbol \mu_n)$ and every $(t, x) \in \R_+ \times \R$,
    \begin{equation}
        \label{eq:sur-tail}
        v_{\r_i}(t,x) \leq u^i(t,x) = v_{\sigma_i}(t,x), \quad i=1,\ldots,n.
    \end{equation}
\end{proposition}

\begin{proof}
By Theorem~\ref{thm::goal}, $u^i(t,x)=v_{\sigma_i}(t,x)$,
so it remains to prove $v_{\rho_i}\le u^i$. We proceed by induction on $i$.

For $i=1$, since $\rho_1\sim\mu_1$, for every $r>0$ and $y\in \R$,
$
    v_{\rho_1}(r,y)
    \le
    \P(\rho_1>r)
    =
    g^1(r).
$
Since \(\mu_1\) is atomless, \(\P(\rho_1>0)=1\), and hence
\(
    v_{\rho_1}(0,y)
    =
    \mathbf 1_{\{y<0\}}.
\)
It follows that, for every \(\tau\in\TT^t\),
\[
\begin{aligned}
    v_{\rho_1}(t-\tau,Y_\tau)
    \le
    \mathbf 1_{\{Y_\tau<0\}}
    +
    \bigl(
        g^1(t-\tau)-\mathbf 1_{\{Y_\tau<0\}}
    \bigr)\mathbf 1_{\{\tau<t\}}.
\end{aligned}
\]
Combining this inequality with Lemma~\ref{lem:opt-ineq} gives
\[
\begin{aligned}
    v_{\rho_1}(t,x)
    &\le
    \E^x\left[
        v_{\rho_1}(t-\tau,Y_\tau)
    \right]
    \le
    \E^x\left[
        \mathbf 1_{\{Y_\tau<0\}}
        +
        \bigl(
            g^1(t-\tau)-\mathbf 1_{\{Y_\tau<0\}}
        \bigr)\mathbf 1_{\{\tau<t\}}
    \right].
\end{aligned}
\]
Taking the infimum over \(\tau\in\TT^t\) yields
\(
    v_{\rho_1}(t,x)\le u^1(t,x).
\)

Now let \(i\ge2\), and suppose that
\(
    v_{\rho_{i-1}}(r,y)\le u^{i-1}(r,y)
\)
for every \((r,y)\in\R_+\times\R\). Since
\(\rho_{i-1}\le\rho_i\), for \(r>0\),
\[
\begin{aligned}
    v_{\rho_i}(r,y)
    &=
    \P(\rho_i>r,\ X_r>y)\\
    &\le
    v_{\rho_{i-1}}(r,y)
    +
    \P(\rho_{i-1}\le r<\rho_i)\\
    &=
    v_{\rho_{i-1}}(r,y)
    +
    g^i(r)-g^{i-1}(r)\\
    &\le
    u^{i-1}(r,y)+w^i(r).
\end{aligned}
\]
At \(r=0\), atomlessness gives
\(
    v_{\rho_i}(0,y)
    =
    u^{i-1}(0,y)
    =
    \mathbf 1_{\{y<0\}}
\)
and
\(
    w^i(0)=0.
\)
Thus, for every \(r\ge0\) and \(y\in\mathbb R\),
\[
    v_{\rho_i}(r,y)
    \le
    u^{i-1}(r,y)
    +
    w^i(r)\mathbf 1_{\{r>0\}}.
\]
Taking \(r=t-\tau\) and combining this inequality with Lemma~\ref{lem:opt-ineq} again, we obtain
\[
\begin{aligned}
    v_{\rho_i}(t,x)
    &\le
    \E^x\!\left[v_{\rho_i}(t-\tau,Y_\tau)\right]
    \le
    \E^x\!\left[
        u^{i-1}(t-\tau,Y_\tau)
        +
        w^i(t-\tau)\mathbf 1_{\{\tau<t\}}
    \right].
\end{aligned}
\]
Taking the infimum over \(\tau\in\TT^t\) and using the recursive definition of \(u^i\) gives
\[
    v_{\rho_i}(t,x)\le u^i(t,x).
\]
The induction is complete.
\end{proof}

Fix \(\rho\in\TT(\boldsymbol\mu_n)\), \(i\in\{1,\ldots,n\}\), and \(t\ge0\). Since \(\rho_i\) and \(\sigma_i\) have the same law, \(\P(\rho_i>t)=\P(\sigma_i>t)\). Moreover, \eqref{eq:sur-tail} gives
\[
    \P(\rho_i>t,\ X_t>x)
    \le
    \P(\sigma_i>t,\ X_t>x),
    \qquad x\in\mathbb R.
\]
Consequently, for every non-decreasing Borel function \(h\) such that \(h(X_t)\mathbf 1_{\{\rho_i>t\}}\) and \(h(X_t)\mathbf 1_{\{\sigma_i>t\}}\) are integrable,
\begin{equation}\label{eq:surviving-state-test-order}
    \E\left[h(X_t)\mathbf 1_{\{\sigma_i>t\}}\right]
    \ge
    \E\left[h(X_t)\mathbf 1_{\{\rho_i>t\}}\right].
\end{equation}
Indeed, the inequality follows directly for non-decreasing step functions. The bounded Borel case follows by approximation, and the integrable case by two-sided truncation.

\begin{proof}[Proof of Theorem \ref{thm:opt-multi}]
Fix \(\rho\in\TT(\boldsymbol\mu_n)\),
\(i\in\{1,\ldots,n\}\), and
\(\theta\in\{\rho_i,\sigma_i\}\). For \(T<\infty\), let
\[
    \gamma_m:=\inf\{t\ge0:|X_t|\ge m\},
    \qquad
    \theta_{T,m}:=\theta\wedge T\wedge\gamma_m.
\]
By It\^o's formula, we have
\[
    F_i(\theta_{T,m},X_{\theta_{T,m}})
    =
    F_i(0,0)
    +
    \int_0^{\theta_{T,m}}
        \mathcal G F_i(t,X_t)\,\mathrm dt+
    \int_0^{\theta_{T,m}}
        \partial_xF_i(t,X_t)\,\mathrm dX_t.
\]
Since \(\partial_xF_i\) is bounded on
\([0,T]\times[-m,m]\), the last integral is a square-integrable
martingale and therefore has expectation zero.

Let
    $M_T:=\sup_{0\le t\le T}|X_t|$.
The polynomial bounds on each finite time interval imply
\[
    |F_i(\theta_{T,m},X_{\theta_{T,m}})|
    \le
    K_{i,T}(1+M_T^{r_{i,T}}) \quad \text{and} \quad
    \int_0^{\theta_{T,m}}
        |\mathcal G F_i(t,X_t)|\,\mathrm dt
    \le
    TK_{i,T}(1+M_T^{r_{i,T}}).
\]
The right-hand sides of both inequalities are integrable. Letting
\(m\to\infty\), dominated convergence and Fubini's theorem yield
\[
\begin{aligned}
    \E[F_i(T\wedge\theta,X_{T\wedge\theta})]
    &=
    F_i(0,0)
    +
    \int_0^T
    \E\!\left[
        \mathcal G F_i(t,X_t)
        \mathbf 1_{\{\theta>t\}}
    \right]\mathrm dt.
\end{aligned}
\]
Applying this identity first with \(\theta=\sigma_i\) and then with
\(\theta=\rho_i\), and taking the difference, we obtain
\begin{equation}\label{eq:stagewise-finite-horizon-gap}
\begin{aligned}
    &\quad\,\,\E[F_i(T\wedge\sigma_i,X_{T\wedge\sigma_i})]
    -
    \E[F_i(T\wedge\rho_i,X_{T\wedge\rho_i})]\\
    &=
    \int_0^T
    \Bigl(
        \E\!\left[
            \mathcal G F_i(t,X_t)
            \mathbf 1_{\{\sigma_i>t\}}
        \right]
        -
        \E\!\left[
            \mathcal G F_i(t,X_t)
            \mathbf 1_{\{\rho_i>t\}}
        \right]
    \Bigr)\,\mathrm dt
    \ge0.
\end{aligned}
\end{equation}
The finite-horizon polynomial bounds make both integrands above integrable.
For almost every \(t\), the function
\(\mathcal G F_i(t,\cdot)\) is non-decreasing, so the last inequality
follows from \eqref{eq:surviving-state-test-order} with
$h=\mathcal GF_i(t,\cdot)$.

Since \(\rho_i\) and \(\sigma_i\) are almost surely finite, for
\(\theta\in\{\rho_i,\sigma_i\}\),
\[
    F_i(T\wedge\theta,X_{T\wedge\theta})
    \longrightarrow
    F_i(\theta,X_\theta)
    \qquad\text{a.s.}
\]
The assumed uniform integrability therefore gives convergence of the
corresponding expectations. Letting \(T\to\infty\) in
\eqref{eq:stagewise-finite-horizon-gap}, we finally obtain
\[
    \E[F_i(\sigma_i,X_{\sigma_i})]
    \ge
    \E[F_i(\rho_i,X_{\rho_i})].
\]
If the uniform-integrability condition holds for all
\(\rho\in\TT(\boldsymbol\mu_n)\), summing over all coordinatewise
inequalities completes the proof.
\end{proof}

\begin{proof}[Proof of Corollary \ref{opt:case1}]
Regard $F_i$ as independent of the time variable. Then
\[
    \mathcal G F_i(t,x)=\frac12F_i''(x)
\]
is non-decreasing in \(x\), and
Theorem~\ref{thm::polygrowth} provides the required polynomial bounds on each finite time interval. For
\(\theta\in\{\rho_i,\sigma_i\}\), let
\(M_\theta:=\sup_{0\le s\le\theta}|X_s|\). Then
\[
    \sup_{T\ge0}|F_i(X_{T\wedge\theta})|
    \le K_i(1+M_\theta^p).
\]
Since \(\theta\sim\mu_i\le_{\mathrm{st}}\mu_n\),
Lemma~\ref{lem:critical-bdg} and the \(p/2\)-moment assumption imply the
uniform integrability. The result now follows from
Theorem~\ref{thm:opt-multi}.
\end{proof}

\begin{proof}[Proof of Corollary \ref{opt:case2}]
Fix \(\rho\in\TT(\boldsymbol\mu_n)\),
\(i\in\{1,\ldots,n\}\), and
\(\theta\in\{\rho_i,\sigma_i\}\).
Let \(M_\theta:=\sup_{0\le s\le\theta}|X_s|\), and set
    $F_i(t,x):=A(t)x$.
Since \(A'\ge0\),
\[
    \mathcal G F_i(t,x)=A'(t)x
\]
is non-decreasing in \(x\). The polynomial bounds on each finite time interval follow because \(A\) and \(A'\) are bounded on every compact time interval. Moreover,
\[
    \sup_{T\ge0}
    |F_i(T\wedge\theta,X_{T\wedge\theta})|
    \le
    K(1+\theta^q)M_\theta.
\]
The right-hand side is integrable by
Lemma~\ref{lem:critical-bdg} and the
\((q+1/2)\)-moment assumption. Hence the stopped payoff families are
uniformly integrable, and Theorem~\ref{thm:opt-multi} yields the result.
\end{proof}

\begin{proof}[Proof of Corollary~\ref{opt:case3}]
Fix \(\rho\in\TT(\boldsymbol\mu_n)\) and
\(i\in\{1,\ldots,n\}\). For
\(\theta\in\{\rho_i,\sigma_i\}\), define
\[
M_\theta:=\sup_{0\le s\le\theta}|X_s|,
\qquad
F_i(t,x):=A(t)F(x).
\]
Since \(A'\ge0\), we have
\[
    \mathcal G F_i(t,x)
    =
    A'(t)F(x)
    +
    \frac12A(t)F''(x).
\]
Both terms are non-decreasing in \(x\). Since \(A\) and \(A'\) are
bounded on every compact time interval,
Theorem~\ref{thm::polygrowth} provides the required polynomial bounds on each finite time interval.

Now, for some constant \(C\),
\[
    \sup_{T\ge0}
    |F_i(T\wedge\theta,X_{T\wedge\theta})|
    \le
    C(1+\theta^q)(1+M_\theta^p).
\]
Lemma~\ref{lem:critical-bdg} and the
\((q+p/2)\)-moment assumption show that the right-hand side is
integrable. Hence the stopped payoff families are uniformly
integrable, and Theorem~\ref{thm:opt-multi} yields the result.
\end{proof}

\appendix

\section{Compactness of the barrier hitting times}
\label{app:lower-barrier-compactness}

This appendix proves the compactness and stability results used in Section~\ref{sec:general_case}. Throughout this appendix, for some
\(k\ge1\), \(\sigma^\xi=\sigma_k\) denotes the \(k\)-th stopping time generated by the recursion in
Section~\ref{sec:proba}. We assume that \(\sigma^\xi\sim\xi\), where
\(\xi\) is an atomless probability measure on \(\mathbb R_+\).
Consequently, \(\sigma^\xi\) is strictly positive and finite almost
surely, and its law is atomless. In addition, note that a closed lower barrier has the form
\[
    \Rc(b)
    :=
    \{(t,x)\in\mathbb R_+\times\mathbb R:x\le b(t)\},
\]
where
    $b:\mathbb R_+\longrightarrow[-\infty,\infty]$
is upper semicontinuous.

\subsection{\texorpdfstring{\(\Gamma\)}{Gamma}-convergence and right discontinuities}

We use the following hypograph analogue of \(\Gamma\)-convergence.
\begin{definition}
\label{def:upper-Gamma-convergence}
Let \(b_n,b:\mathbb R_+\to[-\infty,\infty]\) be upper
semicontinuous. We say that \(b_n\) upper-\(\Gamma\)-converges to
\(b\), and write
\[
    b_n\xrightarrow{\Gamma^+}b,
\]
if the following conditions hold:
\begin{enumerate}
    \item for every sequence \(t_n\to t\),
    $
        \limsup_{n\to\infty}b_n(t_n)\le b(t)$;
    \item for every \(t\ge0\), there exists \(t_n\to t\) such that
    $
        b_n(t_n)\longrightarrow b(t)$.
\end{enumerate}
\end{definition}

For an upper semicontinuous boundary \(b\), define
\[
    \mathcal J_b^+
    :=
    \left\{
        t\ge0:
        \limsup_{s\downarrow t}b(s)<b(t)
    \right\}.
\]

\begin{lemma}
\label{lem:right-discontinuity-countable}
The set \(\mathcal J_b^+\) is at most countable. Consequently,
\(
    \P(\sigma^\xi\in\mathcal J_b^+)=0.
\)
\end{lemma}

\begin{proof}
    For $p, q \in \Q$ with $p < q$ and $m \geq 1$, define
    $$D_{p,q,m}:= \left\{t \geq 0: b(t) > q, \sup_{t<s<t + 1/m}b(s) < p\right\}.$$
    If $t, t' \in D_{p,q,m}$ and $t < t'$, then $t' \geq t + 1/m$; otherwise, $b(t')<p<q<b(t')$, a contradiction. Hence, for $j \geq 0$, each interval $[j/m, (j+1)/m)$ contains at most one point of $D_{p,q,m}$. Thus, $D_{p,q,m}$ is at most countable.

    Now let $t \in \Jc_b^+$. Choose $p, q\in \Q$ such that $\limsup_{s \downarrow t}b(s) < p<q<b(t)$. For some $m\geq 1$, $\sup_{t<s<t+1/m} b(s) < p$, and hence $t \in D_{p,q,m}$. Therefore,
    $$\Jc_b^+ \subseteq \bigcup_{\substack{p, q\in\Q \\ p<q}}\bigcup_{m \geq 1}D_{p,q,m},$$
    which is at most countable.
\end{proof}

\subsection{Hitting times of a shifted process}

Recall that the corresponding stopping regions
$\Rc^1,\ldots,\Rc^k$ are closed and satisfy
$$
    (t,x)\in\Rc^j,\quad y\le x
    \quad\Longrightarrow\quad
    (t,y)\in\Rc^j,
    \qquad j=1,\ldots,k.
$$
For $\eps > 0$, set $X_t^\eps:=X_t-\eps t$, and let $\sigma^{\xi,\eps}$ be the $k$-th stopping time obtained from the same recursion with $X^\eps$ in place of $X$, while keeping $\Rc^1,\ldots,\Rc^k$ unchanged.

\begin{lemma}
\label{lem:shifted-iterated-gate}
For every \(\eps>0\),
\(
    \sigma^{\xi,\varepsilon}\le\sigma^\xi
\)
almost surely.
\end{lemma}

\begin{proof}
Let $\sigma_j^\varepsilon$ denote the $j$-th stopping time in the
recursion driven by $X^\eps$. At the first stopping time,
closedness of $\Rc^1$ gives $(\sigma_1,X_{\sigma_1})\in\Rc^1$.
Since $X_{\sigma_1}^\eps\le X_{\sigma_1}$, we have
$(\sigma_1,X_{\sigma_1}^\varepsilon)\in\Rc^1$. Thus $\sigma_1^\eps\le\sigma_1$.

Suppose that $\sigma_{j-1}^\varepsilon\le\sigma_{j-1}$. Then $\sigma_j\ge\sigma_{j-1}^\varepsilon$, and
$(\sigma_j,X_{\sigma_j})\in\Rc^j$, with $X_{\sigma_j}^\varepsilon\le X_{\sigma_j}$.
Therefore $(\sigma_j,X_{\sigma_j}^\varepsilon)\in\Rc^j$, which implies $\sigma_j^\varepsilon\le\sigma_j$.
Induction gives
$\sigma^{\xi,\varepsilon}\le\sigma^\xi$.
\end{proof}

\begin{lemma}
\label{lem:weak-strict-lower-hit}
For every upper semicontinuous $b: \R_+ \to[-\infty, \infty]$, define
$$\tau_b := \inf\{t \geq \sigma^\xi: X_t \leq b(t)\}, \qquad \tau_b^< := \inf\{t \geq \sigma^\xi : X_t < b(t)\}.$$
Then
\(
    \tau_b=\tau_b^{<}
\)
almost surely.
\end{lemma}

\begin{proof}
The inequality \(\tau_b\le\tau_b^{<}\) holds pathwise. For \(\eps>0\), define
\[
    \tau_b^{<,\varepsilon}
    :=
    \inf\{t\ge\sigma^{\xi,\varepsilon}:X_t^\varepsilon<b(t)\}.
\]
On \(\{\tau_b<\infty\}\), closedness of \(\Rc(b)\) and continuity of \(X\)
imply that $X_{\tau_b}\le b(\tau_b)$.
Since \(\tau_b\ge\sigma^\xi>0\) and $\sigma^{\xi, \eps} \leq \sigma^\xi$ by Lemma~\ref{lem:shifted-iterated-gate}, it follows that
\(
    X_{\tau_b}^\eps
    =
    X_{\tau_b}-\eps\tau_b
    <
    b(\tau_b)
\)
and \(\tau_b^{<,\eps}\leq\tau_b\) almost surely.

Now fix \(T<\infty\). On \(\mathcal F_T\), the law of
\((X_t-\varepsilon t)_{0\le t\le T}\) has density
$Z_T^\eps :=\exp\left(
        -\eps X_T-\frac12\eps^2T
    \right)$
with respect to the Wiener measure. Moreover, $\E[(Z_T^\eps-1)^2] = e^{\eps^2T}-1 \rightarrow 0$. Hence, the Cameron-Martin formula gives
\[
    \P\left(\tau_b^{<,\varepsilon}>T\right)
    {\longrightarrow}
    \P\left(\tau_b^{<}>T\right).
\]
Since, almost surely, $\tau_b^{<,\eps} \le \tau_b \le \tau_b^{<}$, we have
\[
    \P(\tau_b^{<,\eps}>T)
    \le
    \P(\tau_b>T)
    \le
    \P(\tau_b^{<}>T).
\]
Letting \(\eps\downarrow0\) gives
\(
    \P(\tau_b^{<}>T)
    \le
    \P(\tau_b>T)
    \le
    \P(\tau_b^{<}>T).
\)
Thus \(\tau_b\) and \(\tau_b^{<}\) have the same law. Together with
\(\tau_b\le\tau_b^{<}\), this yields \(\tau_b=\tau_b^{<}\) almost surely.
\end{proof}

\subsection{Stability under \texorpdfstring{\(\Gamma\)}{Gamma}-convergence}

\begin{theorem}
\label{thm:delayed-lower-hit-stability}
Let $b_n, b:\R_+ \to [-\infty, \infty]$ be upper semicontinuous and suppose that $b_n \xrightarrow{\Gamma^+}b$.
Then \(
    \tau_{b_n}\longrightarrow\tau_b
\)
almost surely in \([0,\infty]\).
\end{theorem}

\begin{proof}
By Lemmas~\ref{lem:right-discontinuity-countable}
and~\ref{lem:weak-strict-lower-hit}, it is enough to work on an event of
probability $1$ on which $X$ is continuous and
\[
    \sigma^\xi\notin\mathcal J_b^+,
    \qquad
    \tau_b=\tau_b^{<}.
\]
Set $L:= \liminf_{n \to \infty}\tau_{b_n}$. If $L<\infty$, choose a subsequence $(n_j)$ such that $\tau_{b_{n_j}} \rightarrow L$. The closedness of each barrier gives $X_{\tau_{b_{n_j}}} \leq b_{n_j}(\tau_{b_{n_j}})$, and hence $X_L \leq \limsup_{j\to \infty}b_{n_j}(\tau_{b_{n_j}}) \leq b(L)$. Since $L \geq \sigma^\xi$, it follows that $\tau_b \leq L$.

For the reverse bound, fix \(\delta>0\) and define
\[
    \tau^{-\delta}_b
    :=
    \inf\{t\ge\sigma^\xi:X_t\le b(t)-\delta\}.
\]
We next prove
\[
    \limsup_{n\to\infty}\tau_{b_n}\le\tau_b^{-\delta}.
\]
The assertion is clearly valid when \(\tau_b^{-\delta}=\infty\). If $\sigma^\xi<\tau_b^{-\delta}<\infty$,
we have
    $X_{\tau_b^{-\delta}}
    \le
    b(\tau_b^{-\delta})-\delta$.
Choose a sequence \(t_n\to\tau_b^{-\delta}\) such that
    $b_n(t_n)\to b(\tau_b^{-\delta})$.
For all sufficiently large \(n\), \(t_n\ge\sigma^\xi\) and \(X_{t_n}<b_n(t_n)\), so \(
    \tau_{b_n}\le t_n
    \)
    and
    \(\limsup_{n\to\infty}\tau_{b_n}\le\tau_b^{-\delta}\).

Suppose now that \(\tau_b^{-\delta}=\sigma^\xi\). Then $X_{\sigma^\xi}\le b(\sigma^\xi)-\delta$. The case \(b(\sigma^\xi)=-\infty\) is impossible.
If \(b(\sigma^\xi)<\infty\), the fact that
\(\sigma^\xi\notin\mathcal J_b^+\) gives
\(
    \limsup_{r\downarrow\sigma^\xi}b(r)=b(\sigma^\xi).
\)
If \(b(\sigma^\xi)=+\infty\), the same fact says that the right-hand limsup is $\infty$. In either case, continuity of \(X\) therefore implies that, for every \(\eta>0\), there
exists
    $r\in(\sigma^\xi,\sigma^\xi+\eta)$
    such that
    $X_r<b(r)$.
For this fixed \(r>\sigma^\xi\),
take a sequence \(r_n\to r\) with $b_n(r_n) \rightarrow b(r)$. Then
    $r_n\ge\sigma^\xi$ and
    $X_{r_n}<b_n(r_n)$ for sufficiently large $n$. Consequently,
    $$\limsup_{n \rightarrow \infty} \tau_{b_n} \leq r <\sigma^\xi + \eta.$$
Letting $\eta \downarrow 0$ proves the required bound when $\tau_b^{-\delta} = \sigma^\xi$.

Finally, note that
$$\{t \geq \sigma^\xi: X_t < b(t)\} = \bigcup_{k \geq 1}\{t \geq \sigma^\xi: X_t \leq b(t) - 1/k\}.$$
Hence $\tau_b^{-\delta} \downarrow \tau_b^<$ as $\delta \downarrow 0$. Therefore,
$$  \tau_b
    \leq
    \liminf_{n\rightarrow \infty} \tau_{b_n}
    \leq
    \limsup_{n\to\infty}\tau_{b_n}
    \le
    \tau_b^{<}
    =
    \tau_b.
$$
This completes the proof.
\end{proof}

\subsection{Compactness of the lower barriers}

Let $\overline{\T}:=[0,\infty]$ and $\overline{\R}:=[-\infty,\infty]$. Equip these spaces with the topologies induced, respectively, by
$$t \longmapsto \frac{t}{1+t}, \qquad x\longmapsto \frac{2}{\pi}\arctan x,$$
with the maps extended continuously to the endpoints.
Then $K:=\overline{\T}\times\overline{\R}$
is a compact metric space. Fix a compatible metric $d_K$ on $K$, and
let \(d_{\rm H}\) denote the associated Hausdorff metric on the
nonempty compact subsets of \(K\). For an upper semicontinuous \(b\), define its
compactified hypograph by
\[
    \mathsf H(b)
    :=
    \{(t,x)\in[0,\infty)\times\overline{\mathbb R}:x\le b(t)\}
    \cup
    \bigl(\{\infty\}\times\overline{\mathbb R}\bigr).
\]
\begin{proposition}
\label{prop:hypograph-Hausdorff-compactness}
For every sequence of upper semicontinuous functions
\(b_n:\mathbb R_+\to[-\infty,\infty]\), there is a subsequence, still
denoted by \(b_n\), and an upper semicontinuous function
\(b:\mathbb R_+\to[-\infty,\infty]\) such that
\[
    d_{\rm H}\bigl(\mathsf H(b_n),\mathsf H(b)\bigr)
    \longrightarrow0.
\]
In particular,
    $b_n\xrightarrow{\Gamma^+}b$.
\end{proposition}

\begin{proof}
Each \(\mathsf H(b_n)\) is a nonempty compact subset of \(K\). The
space of nonempty compact subsets of a compact metric space is compact
under the Hausdorff metric. Hence, after passing to a subsequence,
\[
    \mathsf H(b_n)\longrightarrow H
\]
for some nonempty compact \(H\subset K\).

The limit \(H\) contains
\(\{\infty\}\times\overline{\mathbb R}\), and every finite-time section
contains \(-\infty\). It is also downward closed in the second coordinate.
Indeed, if \((t,x)\in H\) and \(y<x\), let
\((t_n,x_n)\in\mathsf H(b_n)\) converge to \((t,x)\). For
\(n\) large enough, we have \(y\le x_n\) and hence \((t_n,y)\in\mathsf H(b_n)\). Passing to the
limit, we have \((t,y)\in H\).

For each finite \(t\), the section
    $H_t:=\{x\in\overline{\mathbb R}:(t,x)\in H\}$
is therefore a nonempty compact interval. There is a unique
\(b(t)\in[-\infty,\infty]\) such that
$
    H_t=[-\infty,b(t)]$.
Thus \(H=\mathsf H(b)\), and closedness of \(H\) is equivalent to upper
semicontinuity of \(b\).

It remains to verify upper-\(\Gamma\)-convergence. Fix any $t_n\to t$. Passing to a subsequence, assume $b_n(t_n)\to y:=\limsup_{n \to \infty} b_n(t_n)$ in $\overline{\R}$. As
$(t_n,b_n(t_n))\in\mathsf H(b_n)$, the convergence implies \((t,y)\in\mathsf H(b)\), and hence
\(y\le b(t)\).

Conversely, fix \(t\geq0\). Since \((t,b(t))\in\mathsf H(b)\), Hausdorff convergence implies
$(t_n,x_n)\in\mathsf H(b_n)$ and $(t_n,x_n)\longrightarrow(t,b(t))$.
As \(x_n\le b_n(t_n)\), we have $b(t) \le \liminf_{n \to \infty} b_n(t_n)$. Combining this with the preceding limsup inequality, we conclude that $b_n(t_n)\rightarrow b(t)$.
\end{proof}

\section{Technical Lemmas}

\subsection{Stopping problem reduction}
\label{app:tc-grid-reduction-proof}

\begin{lemma}
\label{lem:tc-grid-rounding}
Fix $t\in[t_0,r_\xi^\beta)$, $x\in\R$, and
$\tau\in\TT^t$. Let
$\mathcal D_t=\{d_0<\cdots<d_q\}$, $d_q=t-t_0$,
and define
\[
\bar\tau
:=
\begin{cases}
\min\{d\in\mathcal D_t:\tau\le d\},
    & \tau\le d_q,\\[2mm]
d_q, & \tau>d_q.
\end{cases}
\]
Then $\bar\tau\in\TT_t^{\Lambda}$ and
\[
\E^x\!\left[
v^\xi(t-\bar\tau,Y_{\bar\tau})
+w^\beta(t-\bar\tau)\mathbf 1_{\{\bar\tau<t\}}
\right]
\le
\E^x\!\left[
v^\xi(t-\tau,Y_\tau)
+w^\beta(t-\tau)\mathbf 1_{\{\tau<t\}}
\right].
\]
\end{lemma}

\begin{proof}
For \(0\le s\le t\), define
$\Psi_s
:=
v^\xi(t-s,Y_s)
+
w^\beta(t-s)\mathbf 1_{\{s<t\}}$.
The process \(\Psi\) is bounded.

We first verify that \(\bar\tau\) is admissible. For
\(j=0,\ldots,q-1\),
$\{\bar\tau\le d_j\}=\{\tau\le d_j\}$,
hence \(\{\bar\tau\le d_q\}=\Omega\). As \(\bar\tau\) takes
values in \(\mathcal D_t\), it follows that
\(
\bar\tau\in\TT_t^{\Lambda}.
\)

We make the following partition of the sample space:
\[
A_0:=\{\tau\le d_0\},\qquad
A_j:=\{d_{j-1}<\tau\le d_j\},
\quad j=1,\ldots,q,
\]
and
$A_*:=\{\tau>d_q\}$.

Fix \(j\in\{1,\ldots,q\}\). The times
\(t-d_j\) and \(t-d_{j-1}\) are consecutive points of \(\Lambda\).
Hence \(g^\beta\) is constant on
\([t-d_j,t-d_{j-1})\), and
\[
\Psi_s
=
g^\beta(t-d_j)
-
\overline v^\xi(t-s,Y_s),
\qquad
\mbox{ for } s\in (d_{j-1}, d_j].
\]
Here the identity at \(s=t\) follows from
\(w^\beta(0)=0\). By
Corollary~\ref{cor::subm-lower-tilde}, \(\Psi\) is a
bounded supermartingale on every closed subinterval of
\((d_{j-1},d_j]\).

For \(0<\varepsilon<d_j-d_{j-1}\), define
$\theta_\varepsilon
:=
\bigl(\tau\vee(d_{j-1}+\varepsilon)\bigr)\wedge d_j$.
Then \(\theta_\varepsilon\) is a stopping time taking values in
\([d_{j-1}+\varepsilon,d_j]\). Moreover,
\(A_j\in\mathcal F^Y_{\theta_\varepsilon}\). Indeed,
\[
A_j\cap\{\theta_\varepsilon\le r\}
=
\begin{cases}
\varnothing,
    &r<d_{j-1}+\varepsilon,\\
\{d_{j-1}<\tau\le r\},
    &d_{j-1}+\varepsilon\le r<d_j,\\
A_j,
    &r\ge d_j.
\end{cases}
\]
Since \(\bar\tau=d_j\) on \(A_j\), by the optional stopping theorem, we have
\[
\E^x\!\left[
    \Psi_{\bar\tau}\mathbf 1_{A_j}
\right]
=
\E^x\!\left[
    \Psi_{d_j}\mathbf 1_{A_j}
\right]
\le
\E^x\!\left[
    \Psi_{\theta_\varepsilon}\mathbf 1_{A_j}
\right].
\]
For every path in \(A_j\), one has
\(\theta_\varepsilon=\tau\) for all sufficiently small
\(\varepsilon\). Since \(\Psi\) is bounded, letting
\(\varepsilon\downarrow0\) yields
\[
\E^x\!\left[
    \Psi_{\bar\tau}\mathbf 1_{A_j}
\right]
\le
\E^x\!\left[
    \Psi_\tau\mathbf 1_{A_j}
\right].
\]

On \(A_0\), one has \(\bar\tau=d_0\). If \(d_0=0\), then
\(\tau=\bar\tau=0\) on \(A_0\). Suppose that \(d_0>0\).
The function \(g^\beta\) is constant on \([t-d_0,t]\), so
Corollary~\ref{cor::subm-lower-tilde} shows that \(\Psi\) is a
bounded supermartingale on \([0,d_0]\). As
$
A_0\in\mathcal F^Y_{\tau\wedge d_0}$
and
$\tau\wedge d_0=\tau$
on $A_0$,
the optional stopping theorem implies that
\[
\E^x\!\left[
    \Psi_{\bar\tau}\mathbf 1_{A_0}
\right]
=
\E^x\!\left[
    \Psi_{d_0}\mathbf 1_{A_0}
\right]
\le
\E^x\!\left[
    \Psi_\tau\mathbf 1_{A_0}
\right].
\]

Finally, recall that \(d_q=t-t_0\) and
\(w^\beta=0\) on \([0,t_0]\). Thus, for \(d_q\le s\le t\),
$
\Psi_s
=
v^\xi(t-s,Y_s)$.
By Lemma~\ref{lem::V-submartingale}, \(\Psi\) is a bounded
submartingale on \([d_q,t]\). Since
\[
A_*=\{\tau>d_q\}\in\mathcal F^Y_{d_q}
\quad\text{and}\quad
\tau\vee d_q=\tau
\quad\text{on }A_*,
\]
the optional stopping theorem implies that
\[
\E^x\!\left[
    \Psi_{\bar\tau}\mathbf 1_{A_*}
\right]
=
\E^x\!\left[
    \Psi_{d_q}\mathbf 1_{A_*}
\right]
\le
\E^x\!\left[
    \Psi_\tau\mathbf 1_{A_*}
\right].
\]
Summing the above inequalities over
\(A_0,A_1,\ldots,A_q,A_*\), we obtain
\[
\E^x[\Psi_{\bar\tau}]
\le
\E^x[\Psi_\tau].
\]
By the definition of \(\Psi\), this is exactly the assertion of the lemma.
\end{proof}
\subsection{Technical estimates for Section~\ref{sec:opti}}
\label{app:optimality-estimates}
We provide the following standard estimates for completeness.
\begin{theorem}
\label{thm::polygrowth}
Suppose that \(p\ge2\), \(F\in C^2(\mathbb R)\), \(F''\) is
non-decreasing, and for all $x \in \R$,
       $ |F(x)|
        \le
        C\left(1+|x|^p\right)
       $.
Then there exists a constant \(C_F<\infty\) such that for all $x \in \R$,
$$
        |F''(x)|
        \le
        C_F\left(1+|x|^{p-2}\right),\qquad
        |F'(x)|
        \le
        C_F\left(1+|x|^{p-1}\right).
$$
\end{theorem}

\begin{proof}
    For every $h > 0$, the monotonicity of $F''$ gives
    $$\frac{F(x) - 2F(x-h) + F(x-2h)}{h^2} \leq F''(x) \leq \frac{F(x+2h) - 2F(x+h) + F(x)}{h^2}.$$
    Choose $h = 1+|x|$. Every argument of $F$ in these finite differences has absolute value at most $3h$. Hence the polynomial-growth assumption yields
    $$|F''(x)| \leq C_1\frac{1+h^p}{h^2} \leq C_2(1+|x|^{p-2}), \quad x \in \R.$$
    Consequently,
    $$|F'(x)| \leq |F'(0)| + \int^{\max\{0, x\}}_{\min\{0, x\}}|F''(y)\d y \le C_3(1+|x|^{p-1}).$$
    Enlarging the constant proves both assertions.
\end{proof}

\begin{lemma}\label{lem:critical-bdg}
Let $\theta$ be an almost surely finite stopping time, and let
    $M_\theta:=\sup_{0\le s\le\theta}|X_s|$.
If \(a,b\ge0\), \(r\ge a+b/2\), \(r>0\), and
\(\E[\theta^r]<\infty\), then
\[
    \E[\theta^aM_\theta^b]
    \le
    C_{a,b}\bigl(1+\E[\theta^r]\bigr).
\]
\end{lemma}

\begin{proof}
The case \(a=b=0\) is immediate. Suppose first that
\(r=a+b/2>0\). When \(a,b>0\), H\"older's inequality and the
Burkholder--Davis--Gundy inequality give, for $N \in \N$,
\[
    \E[(\theta \wedge N)^aM_{\theta \wedge N}^b]
    \le
    \bigl(\E[(\theta\wedge N)^r]\bigr)^{a/r}
    \bigl(\E[M_{\theta\wedge N}^{2r}]\bigr)^{b/(2r)}
    \le
    C_{a,b}\E[(\theta\wedge N)^r].
\]
Letting $N \rightarrow \infty$ gives the desired estimate by monotone convergence. If $a = 0$, the same conclusion follows directly from the Burkholder--Davis--Gundy inequality, and if $b=0$, the conclusion is exactly the assumed moment bound.
Finally, if \(r>a+b/2\), the conclusion follows from
\(\E[\theta^{a+b/2}]\le1+\E[\theta^r]\).
\end{proof}

\end{document}